\documentclass[12pt]{amsart}
\usepackage{amsmath}
\usepackage{amsfonts}
\usepackage{amssymb}
\usepackage{amscd}
\usepackage{mathtools}
\usepackage{amsxtra}
\usepackage{amsthm}
\usepackage{graphicx}
\usepackage[abbrev,alphabetic]{amsrefs}
\RequirePackage[dvipsnames,usenames]{color}
\usepackage{soul,xcolor}
\setstcolor{red}
\usepackage{stmaryrd}
\usepackage{booktabs}
\usepackage{multirow}
\newtagform{tiny}{\tiny(}{)}
\usepackage{threeparttable}

\usepackage{mathtools}
\usepackage{hyperref}
\usepackage[margin=1.25in]{geometry}
\usepackage{mathrsfs}

\usepackage{amsthm}
\usepackage{comment}
\usepackage[all,cmtip]{xy}
\usepackage{tikz-cd}
\usetikzlibrary{cd}

\tikzcdset{
  cells={font=\everymath\expandafter{\the\everymath\displaystyle}},
}

\usepackage[all]{xy}

\usepackage{cleveref}

\makeatletter
\def\@tocline#1#2#3#4#5#6#7{\relax
  \ifnum #1>\c@tocdepth 
  \else
    \par \addpenalty\@secpenalty\addvspace{#2}%
    \begingroup \hyphenpenalty\@M
    \@ifempty{#4}{%
      \@tempdima\csname r@tocindent\number#1\endcsname\relax
    }{%
      \@tempdima#4\relax
    }%
    \parindent\z@ \leftskip#3\relax \advance\leftskip\@tempdima\relax
    \rightskip\@pnumwidth plus4em \parfillskip-\@pnumwidth
    #5\leavevmode\hskip-\@tempdima
      \ifcase #1
       \or\or \hskip 1em \or \hskip 2em \else \hskip 3em \fi%
      #6\nobreak\relax
    \hfill\hbox to\@pnumwidth{\@tocpagenum{#7}}\par
    \nobreak
    \endgroup
  \fi}
\makeatother

\renewcommand{\mod}{\ \textrm{mod}\ }
 
\renewcommand{\P}{\mathbb{P}}
\newcommand{\Z}{\mathbb{Z}}

\newcommand{\F}{\mathbb{F}}

\newcommand{\cG}{\mathcal{G}}

\newcommand{\cL}{\mathcal{L}}

\newcommand{\cO}{\mathcal{O}}

\newcommand{\dR}{\mathrm{dR}}

\newcommand{\sO}{\mathcal{O}}

\newcommand{\m}{\mathfrak{m}}

\newcommand{\Frac}{\mathrm{Frac}}
\DeclareMathOperator{\Proj}{\mathrm{Proj}}

\newcommand{\wt}{\widetilde}

\DeclareMathOperator{\Sym}{Sym}

\newcommand{\ns}{\mathrm{ns}}

\DeclareMathOperator{\Spec}{Spec}

\DeclareMathOperator{\Hom}{Hom}

\DeclareMathOperator{\Ext}{Ext}

\DeclareMathOperator{\Aut}{Aut}

\DeclareMathOperator{\PGL}{PGL}
\DeclareMathOperator{\Pic}{Pic}

\DeclareMathOperator{\Ker}{Ker}

\DeclareMathOperator{\sht}{ht}
\DeclareMathOperator{\NS}{NS}
\DeclareMathOperator{\disc}{disc}
\DeclareMathOperator{\rank}{rank}

\DeclareMathOperator{\ppt}{ppt}

\newcommand{\bs}{\boldsymbol{s}}

\renewcommand{\div}{{\rm div}}
\renewcommand{\Im}{\mathrm{Im}}

\theoremstyle{plain}
\newtheorem{theorem}{Theorem}[section]

\newtheorem{proposition}[theorem]{Proposition}

\newtheorem{lemma}[theorem]{Lemma}

\newtheorem{corollary}[theorem]{Corollary}

\newtheorem*{claim*}{Claim}

\newtheorem{theoremA}{Theorem}

\theoremstyle{definition}
\newtheorem{definition}[theorem]{Definition}

\newtheorem{example}[theorem]{Example}
\newtheorem{notation}[theorem]{Notation}

\newtheorem*{setup*}{Setup}

\theoremstyle{remark}
\newtheorem{remark}[theorem]{Remark}
\newtheorem*{ackn}{Acknowledgements}

\theoremstyle{plain}

\numberwithin{equation}{section}

\crefname{theorem}{Theorem}{Theorems}
\crefname{proposition}{Proposition}{Propositions}
\crefname{lemma}{Lemma}{Lemmas}
\crefname{corollary}{Corollary}{Corollaries}
\crefname{conjecture}{Conjecture}{Conjectures}
\crefname{claim}{Claim}{Claims}
\crefname{notation}{Notation}{Notations}
\crefname{remark}{Remark}{Remarks}
\crefname{example}{Example}{Examples}
\crefname{definition}{Definition}{Definitions}
\crefname{theoremA}{Theorem}{Theorems}

\title[Formula for the Artin invariant of smooth K3 hypersurfaces]{An explicit formula for the Artin invariant of smooth K3 hypersurfaces}

\author{Teppei Takamatsu}
\address{Department of Mathematics, Faculty of Science,
Saitama University,
255 Shimo-Okubo, Sakura-ku,
Saitama-shi, Saitama 338-8570,
Japan}
\email{teppeitakamatsu.math@gmail.com}

\author{Shou Yoshikawa}
\address{Institute of Science Tokyo, Tokyo 152-8551, Japan}
\email{yoshikawa.s.9fe9@m.isct.ac.jp}

\begin{document}

\begin{abstract}
We characterize the Artin invariant of a smooth K3 hypersurface in terms of quasi-$F$-splitting. As an application, we obtain an explicit formula for this invariant.
\end{abstract}

\maketitle

\tableofcontents
\section{Introduction}
\label{Section:intro}

Let \(k\) be an algebraically closed field of characteristic \(p\), and let \(X\) be a K3 surface over \(k\).
Associated with \(X\) are two invariants specific to positive characteristic: the height and the Artin invariant.

The height \(h(X)\) of \(X\) is defined as the height of the formal Brauer group of \(X\), and one has
\[
h(X) \in \{1,2,\dots,10,\infty\}.
\]
If \(h(X)=\infty\), then \(X\) is said to be supersingular.
In this case, there exists an integer
\[
\sigma(X) \in \{1,2,\dots,10\} \quad \textup{such that} \quad 
\disc \NS(X) = -p^{2\sigma(X)},
\]
and \(\sigma(X)\) is called the Artin invariant of \(X\).
The Artin invariant is a fundamental invariant of supersingular K3 surfaces; in particular, it plays a basic role in the crystalline Torelli theory for supersingular K3 surfaces (\cite{Ogus-supersingularK3crystal}, \cite{Ogus-Crystalline-Torelli}).

The height and the Artin invariant are also important in that they give rise to stratifications on the moduli stack \(\mathcal{F}_{2D}\) of primitively polarized K3 surfaces of degree \(2D\):
\[
\mathcal{F}_{2D, i}=\{(X,L)\mid h(X)\ge i\},
\qquad
\mathcal{F}_{2D, \infty}^{j}=\{(X,L)\mid h(X)=\infty,\ \sigma(X)\le j\}.
\]
These stratifications have been studied from several points of view, especially when \(p\nmid D\), including their cycle classes and their relation to \(F\)-zips (\cite{vdGeer-Katsura}, \cite{Ekedahl-vdGeer}).
On the other hand, in \cite{OgusK3}, Ogus introduced an invariant \(\tau(X,L)\), which refines \(\sigma(X)\) and coincides with \(\sigma(X)\) when \(p\nmid D\), and studied in detail the structure of these stratifications, including the case \(p\mid D\).

However, it has been difficult to compute the height and the Artin invariant explicitly for a given K3 surface (cf.~Remark~\ref{rmk:computation-Artin-invariant}).
For the height, one approach is via point counting over finite fields (cf.\ \cite{Yu-Yui}*{Section~4} and \cite{Taelman}*{Theorem~1};
see also \cite{Matsumoto}*{Proposition~5.10}).
On the other hand, in special cases, including sextic purely inseparable double planes in characteristic 2 and Delsarte surfaces, explicit formulas for Artin invariants are also available (\cite{Shimadachar2algo}, \cite{Goto04} and \cite{Yui}).

More recently, in the case of smooth quartic surfaces in \(\mathbb{P}^3\), a simple and explicit formula was obtained in \cite{kty} (cf.\ \cite{Fed}, \cite{Stienstra}).
This formula is built on Yobuko's theorem that the quasi-\(F\)-split height coincides with the height.
Here, the quasi-\(F\)-split height is an invariant which generalizes and quantifies the classical notions of ordinarity and \(F\)-splitting.

Thus, Yobuko's result may be regarded as a bridge between the height of a K3 surface and the theory of \(F\)-singularities.
The aim of this paper is to establish a similar connection between the Artin invariant and \(F\)-singularities for smooth quartic surfaces, and to give an explicit formula for computing the Artin invariant in this setting.

For a homogeneous element $f\in A:=k[x,y,z,w]$ of degree $4$, we define its \emph{non-splitting index} by
\[
\ns(f):=
\inf\left\{
n\in \mathbb Z_{\ge 1}
\;\middle|\;
\text{$(A,(1-\frac{2}{p^n})\operatorname{div}(f))$ is not $n$-quasi-$F$-split}
\right\},
\]
where we set $\inf \emptyset:=\infty$.
Our first main result shows that, for supersingular quartic K3 surfaces $X \subset \P^3$, this invariant determines $\tau(X,\cO(1))$.

\begin{theoremA}[\Cref{ns-vs-K3-tau}]\label{intro:AI-to-ns}
For $f \in A$ as above, we assume that $X=\Proj(A/fA)$ is a supersingular K3 surface.
Then 
\[
\tau(X,\cO(1))=
\begin{cases}
\ns(f) & \text{if $\ns(f)\le 9$},\\
10 & \text{if $\ns(f)\ge 10$}.
\end{cases}
\]
\end{theoremA}

Note that $\tau (X, \cO(1)) = \sigma (X)$ if $p \neq 2$ or $X$ contains a line.
Moreover, we prove a similar theorem for smooth sextic hypersurfaces in the weighted projective space $\P_k(1,1,1,3),$ in which case $\tau (X, \cO(1)) = \sigma (X)$ always holds.
Since $\ns (f)$ is defined for an arbitrary hypersurface,
Theorem~\ref{intro:AI-to-ns} suggests the possibility of extending the Artin invariant, or its analogue, to broader classes of varieties, such as higher-dimensional Calabi--Yau hypersurfaces and other classes of varieties in positive characteristic.

The case $\ns(f)=1$ is closely related to the result of Bhatt--Singh \cite{Bhatt-Singh}.
Indeed, by \cite{Bhatt-Singh}*{Theorem~1.1}, when $p\ge 5$ the condition $\ns(f)=1$ is equivalent to the statement that the Hasse invariant vanishes to order $2$ on the versal deformation space of $X\subset \mathbb P_k^3$.
This is in turn equivalent to the condition that the Artin invariant of $X$ is equal to $1$ (cf.\ \cite[Theorem 1]{OgusHasse} and \cite[Proposition 2.4]{ItoBrauer}).
Thus, Theorem~\ref{intro:AI-to-ns} may be viewed as a higher-level analogue of \cite{Bhatt-Singh}*{Theorem~1.1}.

A key ingredient in the proof of Theorem~\ref{intro:AI-to-ns} is an explicit description of the moduli stratification
\[
 \mathcal{F}_{2D, 10} \subset \cdots \subset \mathcal{F}_{2D, 1} = \mathcal{F}_{2D},
\]
which also yields a concrete formula for the Artin invariant in terms of the defining equation of the quartic surface.

More precisely, for a quartic polynomial \(f\), we construct explicitly a \(35\times 1\) column vector \(v_f\), a \(1\times 35\) row vector \(\lambda\), and a \(35\times 35\) matrix \(T\) from the coefficients of \(f\) and the Frobenius action (see Section~\ref{section:explicit-formula} for the precise definitions). 
To keep the formula in the introduction simple, we state it only in the case where \(f\) is a polynomial over \(\mathbb F_p\). 
Combining Theorem~\ref{intro:AI-to-ns} with the following theorem, we obtain the explicit formula for \(\tau(X,\mathcal O(1))\), which is well suited to computer algebra computations (see \cite{Takamatsu_code}).

\begin{theoremA}[\Cref{explicit-formula}]\label{intro:comp}
In the above setting, we have
\[
\sht(X)=\inf\{\,n\ge 1\mid \lambda T^{n-1} v_f \neq 0\,\},
\]
where we set $\inf \emptyset:=\infty$.

Furthermore, if $\lambda T^{n-1}v_f=0$ for every $n \ge 1$, then
\[
\ns(f)
=
\inf\left\{\,n\ge 1\ \middle|\ 
\operatorname{rank}
\begin{pmatrix}
\lambda\\
\lambda T\\
\vdots\\
\lambda T^{n-1}
\end{pmatrix}
\le n-1
\right\}.
\]
\end{theoremA}

We also prove analogous comparison results for weighted hypersurfaces and for higher-dimensional Calabi--Yau hypersurfaces.
Thus, Theorem~\ref{intro:comp} is part of a more general pattern in which Frobenius-theoretic invariants of hypersurfaces admit explicit descriptions in terms of linear algebra attached to the defining equation.

Using the splitting-order sequence introduced in \cite{Yoshikawa25-cri}, we define a mixed-characteristic analogue of the non-splitting index.
This should be viewed as an extension of the mixed-characteristic quasi-$F$-split height studied in \cites{Yoshikawa25,OTY25}.
More precisely, those works prove that, for Calabi--Yau hypersurfaces, the Artin--Mazur height in positive characteristic coincides with the quasi-$F$-split height of a lift in mixed characteristic.
In the same spirit, we show that the non-splitting index can also be compared in positive and mixed characteristic.
By comparing the positive- and mixed-characteristic non-splitting indices, we obtain the following result, which generalizes an argument from \cite{TY26}.

\begin{theoremA}[\Cref{AI-to-example}, cf.~\cite{TY26}]\label{intro:AI-to-example}
Let $X:=(f=0)\subset \mathbb P^3$ be a quartic K3 surface over an algebraically closed field $k$ of characteristic $p>0$.
If $X$ is supersingular with Artin invariant $n$, then the ring
\[
k[x,y,z,w,t]/(t^{p^n}+f)
\]
is not quasi-$F$-split.
\end{theoremA}

Applying this comparison to weighted sextic K3 surfaces in characteristic two, we obtain the following constraint on the Artin invariant, which is an analogue of Degtyarev's observation (\cite{Degtyarev}*{Remark 7.7}) on quartic surfaces.

\begin{theoremA}[\Cref{no-sigma-geq-3}]
Let $X$ be a smooth sextic hypersurface in $\mathbb P_k(1,1,1,3)$ in characteristic two.
If $X$ is supersingular, then $\sigma(X)\ge 3$.
\end{theoremA}

The paper is organized as follows.
In \Cref{Section:Fedder}, we recall and reformulate the Fedder-type criteria for quasi-$F$-splitting.
In \Cref{Section:nspositive}, we study the non-splitting index in positive characteristic.
We first establish a Fedder-type criterion for computing the non-splitting index, and then prove \Cref{intro:comp}.
In \Cref{Section:Ogus}, we recall results of Ogus on K3 surfaces and their moduli.
We then study families of smooth K3 hypersurfaces.  We compare the non-splitting index with $\tau(X,\cO(1))$ for quartic and weighted sextic K3 surfaces, and prove \Cref{intro:AI-to-ns}.
In \Cref{section:mixed}, we introduce the non-splitting index in mixed characteristic via the splitting-order sequence.
We compare the positive- and mixed-characteristic non-splitting indices, and also discuss the possible values of the non-splitting index for lifts (\Cref{compare-ns-positive-mixed}).
Finally, in Section~\ref{Section:example}, we give several explicit examples. 
We exhibit equations of K3 surfaces with Artin invariant $3,4,\dots,9$ in characteristic $2$  and with Artin invariant $1,2,\dots,10$ in characteristic $3$. 
We also construct a smooth Calabi–Yau threefold whose defining equation has non-splitting index 58, and we generalize Shioda--Goto’s formula for the Artin invariant of Delsarte K3 surfaces (\Cref{Delsarte}).

\begin{remark}\label{rmk:computation-Artin-invariant}
The Artin invariant has also been computed and studied for several special classes of K3 surfaces.
For example, \cite{Shiodaexplicit} and \cite{Goto96} gave explicit formulas for K3 surfaces of (weighted) Delsarte type.
Moreover, \cite{Rudakov-Shafarevich} constructed, in characteristic \(2\), families of supersingular K3 surfaces realizing each possible Artin invariant with explicit Weierstrass equations.
In addition, \cite{Shiodalecturenote} proves the existence of supersingular K3 surfaces with all the possible Artin invariants for $p>2$ (without explicit equations).
Furthermore, \cite{Shimadachar2algo}*{Algorithm 9.4} (cf.\ \cite{Shimadachar2} and \cite{ShimadaModulicurves})  gives an algorithm for computing the Artin invariant of supersingular K3 surfaces in characteristic 2 from their birational models given as a sextic purely inseparable double plane (cf.\ \cite{Blass-Lang}).
On the other hand, \cite{Ito-supersingularreduction} calculates heights and Artin invariants of the reduction modulo $p$ of K3 surfaces with complex multiplication.
Also, see  \cite{ShimadaSupersingularodd}, \cite{Shimadaprojectivemodelchar5},  \cite{Shimada-Zhang}, \cite{Shimada-ZhangKummer}, \cite{KatsuraSchutt} and \cite{Pho-Shimada} for related results.
\end{remark}

\begin{ackn}
The authors are grateful to Yuya Matsumoto for helpful comments.
The first author was supported by JSPS KAKENHI Grant Number JP25K17228.
The second author was supported by JSPS KAKENHI Grant number JP24K16889.
\end{ackn}

\section{Fedder-type criterion for quasi-$F$-splitting}
\label{Section:Fedder}

In this section, we recall the Fedder-type criterion for quasi-$F$-splitting developed in \cites{kty,KTY2}, together with the notation involved, and slightly reformulate it in a form suitable for our applications.

\begin{notation}\label{conv:local}
Let $(A,\m)$ be an $F$-finite regular local ring of characteristic $p>0$, and let $f \in A$ be a nonzero element.
Let
\[
\Delta \colon A \to A/F(A)
\]
denote the map $\Delta_1$ defined in \cite{kty}*{Section~3.1}.
Let $u \in \Hom_A(F_*A,A)$ be a generator of $\Hom_A(F_*A,A)$ as an $F_*A$-module.
We define an $A$-module homomorphism
\[
\theta_f \colon \Ker(u) \to A,
\qquad
a \mapsto f^{p-2}u(F_*(\Delta(f)a)).
\]
We then define an increasing sequence $\{I_n(f)\}_{n\ge1}$ of ideals by
\[
I_1(f):=f^{p-1}A
\]
and
\[
I_{n+1}(f):=\theta_f(F_*I_n(f)\cap \Ker(u))+f^{p-1}A
\]
inductively.

Note that
\[
\Delta(f^{p-1})=- f^{p(p-2)}\Delta(f) \quad \text{in } A/F(A)
\]
by \cite{kty}*{Proposition~3.8(4)}.
Hence, the homomorphism $\theta_f$ is the negative of the map $\theta$ defined in \cite{kty}, and in particular the ideals $I_n(f)$ coincide with the ideals $I_n$ in \cite{kty}.
Therefore, by \cite{kty}*{Theorem~A}, the quasi-$F$-splitting height of $A/fA$ is given by
\[
\sht(A/fA)=\inf\{\,n\ge1 \mid I_n(f)\nsubseteq \m^{[p]}\,\},
\]
where we set $\inf \emptyset := \infty$.
\end{notation}

\begin{notation}\label{conv:polynomial}
Let $k$ be a perfect field, and set
\[
S:=k[x_0,\ldots,x_N].
\]
We define a map
\[
\Delta \colon S \to S
\]
as follows.
For $a\in S$, write
\[
a=\sum_{i=1}^m c_iM_i
\]
as its monomial decomposition, where the $M_i$ are distinct monomials with coefficients 1 and $c_i\in k$.
We then define
\[
\Delta(a):=
\sum_{\substack{0 \leq \alpha_1, \ldots,\alpha_m \leq p-1 \\
\alpha_1+\cdots+\alpha_m=p}}
\frac{1}{p}\binom{p}{\alpha_1, \ldots ,\alpha_m}
(c_1M_1)^{\alpha_1}\cdots(c_mM_m)^{\alpha_m}.
\]
As in \cite{kty}*{Convention~5.1}, this fits into the commutative diagram
\[
\begin{tikzcd}
    S \arrow[r,"\Delta"] \arrow[d] & S \arrow[r] & S/F(S) \arrow[d] \\
    A \arrow[rr,"\Delta"] & & A/F(A),
\end{tikzcd}
\]
where $A:=S_{(x_0,\ldots,x_N)}$.

Let $u \in \Hom_S(F_*S,S)$ be the homomorphism defined in \cite{kty}*{Convention~5.1}, and define
\[
\theta_f \colon F_*S \to S,
\qquad
a \mapsto  f^{p-2}u(F_*(\Delta(f)a)).
\]

Now assume that $S$ is $\mathbb Z$-graded with $\deg(x_i)>0$ for all $i$, and let $f\in S$ be a homogeneous element satisfying
\[
\deg(f)=\deg(x_0)+\cdots+\deg(x_N).
\]
Since
\[
\Delta(f^{p-1}) \equiv -f^{p(p-2)}\Delta(f) \pmod{F(S)},
\]
it follows from \cite{kty}*{Theorem~C} that
\begin{align*}
\sht(S/fS)=
\inf\{\,n\ge1 \mid \theta_f^{\,n-1}(F_*^{\,n-1}f^{p-1}) \notin \m^{[p]}\,\},
\end{align*}
where $\m:=(x_0,\ldots,x_N)$.

Furthermore, by \cite{KTY2}*{Proposition~3.9 and Remark~3.11(1)}, if $S$ is a $\Z$-graded ring with $\deg (x_i) >0$ for all $i$ such that $\Proj(S)$ is a well-formed weighted projective space and $N\ge2$, then
\[
\sht(\Proj(S/fS))=\sht(S/fS)
\]
for any homogeneous element $f \in S$ such that the smooth locus of $\Proj (S)$ contains any codimension one point of $\Proj (S/fS)$.
We note that if $\cO_{\Proj(S)}(-\deg(f))$ is invertible and $\Proj (S/fS)$ is normal, then the final condition is automatically satisfied.
\end{notation}

\begin{proposition}\label{variant-fedder}
We use the notation in Convention~\ref{conv:polynomial}.
Now assume that $S$ is $\mathbb Z$-graded with $\deg(x_i)>0$ for all $i$, and let $f\in S$ be a homogeneous element satisfying
\[
\deg(f)=\deg(x_0)+\cdots+\deg(x_N).
\]
For a positive integer $n$ and $a \in k$, if $S/fS$ is not $n$-quasi-$F$-split, then
\[
\theta^{n-1}_f(F^{n-1}_*(f^{p-1}a)) \in \Ker(u) \cap \m^{[p]}.
\]
Furthermore, we have
\[
\sht(S/fS)
=
\inf\{\,n\ge1 \mid f^{p-1}(f^{p(p-2)}\Delta(f))^{1+\cdots+p^{n-2}} \notin \m^{[p^n]}\,\}.
\]
\end{proposition}

\begin{proof}
By \cite{kty}*{Theorem~C}, we have
\[
\theta^{n-1}_f(F^{n-1}_*(f^{p-1}a)) \in \m^{[p]}.
\]
Since the degree of $\theta^{n-1}_f(F^{n-1}_*(f^{p-1}a))$ is $(p-1)\deg(f)$, the first assertion follows from \cite{kty}*{Proposition~3.9}.

Next, we prove the final assertion.
Since the degree of
\[
f^{p-1}(f^{p(p-2)}\Delta(f))^{1+\cdots+p^{n-2}}
\]
is $(p^n-1)\deg(f)$, the element is contained in $\m^{[p^n]}$ if and only if
\[
u^{n-1}(F^{n-1}_*(f^{p-1}(f^{p(p-2)}\Delta(f))^{1+\cdots+p^{n-2}}))=\theta^{n-1}_f(F^{n-1}_*f^{p-1}) \in \m^{[p]}
\]
by \cite{kty}*{Proposition~3.9}.
Thus, the final assertion follows from \cite{kty}*{Theorem~C}.
\end{proof}

\section{Non-splitting index in positive characteristic}
\label{Section:nspositive}

\subsection{Fedder-type method for computing a non-splitting index}
\label{subsection:Fedder-type-method}
\begin{definition}\label{ns-positive}
Let $(A,\m)$ be an $F$-finite regular local ring of characteristic $p>0$, and let $f \in A$ be a nonzero element.
We define the \emph{non-splitting index} $\ns(f)$ of $f$ by
\[
\ns(f):=
\inf\Bigl\{ n \in \Z_{\geq 1} \ \Bigm| \ \text{$\bigl(A,\bigl(1-\frac{2}{p^n}\bigr)\operatorname{div}(f)\bigr)$ is not $n$-quasi-$F$-split} \Bigr\}
\]
if this set is nonempty, and set $\ns(f):=\infty$ otherwise.
\end{definition}

\begin{remark}\label{rmk:qFs}
Let $A$ and $f$ be as in \Cref{ns-positive}.
Let
\[
D_n:=\left(1-\frac{2}{p^n}\right)\mathrm{div}(f).
\]
We define a $W_n(A)$-module $Q_{D_n,n}$ by the following pushout diagram:
\[
\begin{tikzcd}
W_n(A) \arrow[r, "F"] \arrow[d, "\mathrm{Res}"] & F_*W_n(A) \arrow[r] & F_*W_n(A)(pD_n) \arrow[d] \\
A \arrow[rr] & & Q_{D_n,n}.
\end{tikzcd}
\]
Then the pair $(A,D_n)$ is $n$-quasi-$F$-split if and only if the evaluation map
\[
\Hom_A(Q_{D_n,n},A) \to A
\]
is surjective.
We refer to \cite{KTTWYY1} for details on quasi-$F$-splitting of pairs.
\end{remark}

\begin{theorem}\label{Fedder-ns}
Let $(A,\m)$ be an $F$-finite regular local ring of characteristic $p>0$, and let $f \in A$ be a nonzero element.
We use the notation introduced in Convention~\ref{conv:local}.
Define a sequence of ideals $\{I^{\ns}_n(f)\}_{n \geq 1}$ by
\begin{itemize}
    \item $I^{\ns}_1(f):=f^{p-2}A$;
    \item $I^{\ns}_{n+1}(f):=\theta_f\bigl(F_*(I_n^{\ns}(f) \cap \Ker(u))\bigr)+f^{p-1}A$.
\end{itemize}
Then
\[
\ns(f)=\inf\{n \geq 1 \mid I^{\ns}_n(f) \subseteq \m^{[p]}\},
\]
where the right-hand side is defined to be $\infty$ if the set is empty.
\end{theorem}

\begin{proof}
We set
\[
D_n:=\left(1-\frac{2}{p^n}\right)\mathrm{div}(f).
\]
It suffices to show that the image of the evaluation map
\[
\Hom_A\bigl(Q_{D_n,n},A\bigr) \to A
\]
coincides with $u(F_*I^{\ns}_n(f))$.

Since an element $(a_0,\ldots,a_{n-1}) \in W_n(\Frac(A))$ belongs to $W_n(pD_n)$ if and only if
\[
a_i \in (f^{p^{i+1}-1})^{-1}A \quad \text{for } 0 \leq i \leq n-2,
\qquad
a_{n-1} \in (f^{p^{{n}}-2})^{-1}A,
\]
a homomorphism $Q_{D_n,n} \to A$ corresponds to a homomorphism
\[
\psi \in \Hom_A(Q_{A,n},A)
\]
such that
\begin{equation}\label{eq:psi}
\psi(a_0f,a_1f,\ldots,a_{n-2}f,a_{n-1}f^2) \in fA
\quad
\text{for every } a_0,\ldots,a_{n-1} \in A.
\end{equation}

Thus, by the proof of \cite{kty}*{Theorem~4.11} together with \cite{kty}*{Lemma~4.10}, the condition \eqref{eq:psi} is equivalent to the existence of elements $h_1,\ldots,h_n \in A$ such that $h_n \in f^{p-2}A$, $h_i \in \Ker(u)$ for $i \geq 2$, 
\[
h_s-\theta_f(F_*h_{s+1}) \in f^{p-1}A \quad \text{for $1 \leq s \leq n-1$},
\]
and $\psi(1)=u(F_*h_1)$.
Therefore,
\[
\Im\!\left(\Hom_A\bigl(Q_{D_n,n},A\bigr)\to A\right)=u(F_*I^{\ns}_n(f)),
\]
as desired.
\end{proof}

\begin{proposition}\label{finite-ns-to-non-qfs}
Let $(A,\m)$ be an $F$-finite regular local ring of characteristic $p>0$, and let $f \in A$ be a nonzero element.
If $\ns(f)<\infty$, then $A/f$ is not quasi-$F$-split.
\end{proposition}

\begin{proof}
We use the notation introduced in Convention~\ref{conv:local}.
By \cite{kty}*{Remark~4.12}, the sequence $\{I_n(f)\}_{n \geq 1}$ is increasing:
\[
I_1(f) \subseteq I_2(f) \subseteq \cdots \subseteq I_n(f) \subseteq \cdots.
\]

For every $a \in \Ker(u)$,
\[
\theta_f(F_*a)
= f^{p-2}u(F_*(\Delta(f)a))
\in f^{p-2}A.
\]
Therefore,
\[
I_2^{\ns}(f) \subseteq I_1^{\ns}(f).
\]
Thus, the sequence $\{I_r^{\ns}(f)\}_{r \geq 1}$ is decreasing:
\[
I_1^{\ns}(f) \supseteq I_2^{\ns}(f) \supseteq \cdots \supseteq I_r^{\ns}(f) \supseteq \cdots.
\]

Since we have
\[
I_r(f) \subseteq I_r^{\ns}(f)
\]
for every $r \geq 1$, 
we obtain 
\[
I_n (f) \subset I_{\max (n,m)} (f) \subset I^{\ns}_{\max(n,m)} (f) \subset I_{m}^{\ns} (f)
\]
for any $n, m \geq 1$.

Now set $N:=\ns(f)<\infty$.
By definition, we have $I_N^{\ns}(f) \subseteq \m^{[p]}.$
Hence, for every $n \geq 1$, we have
\[
I_n(f) \subseteq I_N^{\ns}(f) \subseteq \m^{[p]}.
\]
Therefore, $A/f$ is not quasi-$F$-split by \cite{kty}*{Theorem~A}.
\end{proof}

\subsection{Moduli-theoretic interpretation of non-splitting index}
In this subsection, we compare the non-splitting index with the Jacobian rank of polynomials associated with the height stratification of the moduli of Calabi--Yau hypersurfaces.

\begin{notation}\label{notation:AI-vs-NS}
Let $k$ be an algebraically closed field of characteristic $p>0$
and $A:=k[x_0,\ldots,x_N]$ a polynomial ring with the graded structure defined by $\deg(x_i)=q_i$ for $0 \leq i \leq N$.
Set $d:=q_0+\cdots+q_N$.
We fix the following notation.

\begin{itemize}
\item
Set
\[
\wt{A}:=W(k)[x_0,\ldots ,x_N],
\]
and put $x:=x_0\cdots x_N$.

\item
Define
\[
\Lambda
:=\{\alpha=(a_0,\ldots,a_N)\in \Z_{\ge 0}^{N+1} \mid q_0a_0+\cdots+q_Na_N=d\}.
\]
Fix an ordering on $\Lambda$ and write
\[
\Lambda=\{\alpha_1,\alpha_2,\ldots,\alpha_m\},
\qquad m:=|\Lambda|.
\]
For $\alpha_i=(a_{i,0},\ldots,a_{i,N})\in\Lambda$, we set
\[
x^{\alpha_i}:=x_0^{a_{i,0}}\cdots x_N^{a_{i,N}}.
\]

\item
Let
\[
s=(s_1,\ldots,s_m), \qquad t=(t_1,\ldots,t_m)
\]
be tuples of indeterminates.
\item
Define
\[
G(s):=\sum_{i=1}^m s_i x^{\alpha_i}
\in A[s_1,\ldots,s_m],
\]
and
\[
\wt{G}(s,t)
:=\sum_{i=1}^m (s_i+pt_i)x^{\alpha_i}
\in \wt{A}[s_1,\ldots,s_m,t_1,\ldots,t_m].
\]

\item
Let
\[
\phi\colon \wt{A}[s_1,\ldots,s_m,t_1,\ldots,t_m]
\rightarrow
\wt{A}[s_1,\ldots,s_m,t_1,\ldots,t_m]
\]
be the ring homomorphism defined by
\[
\phi|_{W(k)}=F,\qquad
\phi(x_j)=x_j^p \ (0\le j\le N),\qquad
\phi(s_i)=s_i^p,\ \phi(t_i)=t_i^p.
\]
We define $\Delta(G)(s,t)\in A[s_1,\ldots,s_m,t_1,\ldots,t_m]$
as the image of
\[
\frac{1}{p}\bigl(\wt{G}(s,t)^p-\phi(\wt{G}(s,t))\bigr)  \in \widetilde{A}[s_1,\ldots,s_m,t_1,\ldots,t_m]
\]
under the natural reduction map
\[
\wt{A}[s_1,\ldots,s_m,t_1,\ldots,t_m]
\rightarrow
A[s_1,\ldots,s_m,t_1,\ldots,t_m].
\]
Note that
\[
\Delta(G)(s,t)=\Delta(G)(s,0)-G(t)^p.
\]

\item For $b,c \in k^m$, we set
\[
G(b)=\sum_{i=1}^m b_ix^{\alpha_i} \in A,
\]
which is homogeneous of degree $d$.
We note that $\Delta (G) (b,c) \in A$ agrees modulo $F(A)$ with $\Delta (G(b))$ in Subsection \ref{subsection:Fedder-type-method}.
Furthermore, we set
\[
\wt{G}(b,c):=\sum_{i=1}^m ([b_i]+p[c_i])x^{\alpha_i} \in \wt{A},
\]
which is a lift of $G(b)$.
Then the image of 
\[
\Delta(\wt{G}(b,c)) := 
\frac{1}{p}\bigl(\wt{G}(b,c)^p-\phi|_{\wt{A}}(\wt{G}(b,c))\bigr)  \in \widetilde{A}
\]
in $A$ is $\Delta(G)(b,c) \in A$.

\item
Set
\[
G_1(s,t):=G(s)^{p-2},
\]
and for $n\ge 2$, define
\[
G_n(s,t):=G(s)^{p^n-p^{n-1}-\cdots-p-2}\,
\Delta(G)(s,t)^{p^{\,n-2}+\cdots+p+1}.
\]

\item 
For $n \geq 1$, let
\[
H_n(s,t)\in A[s_1,\ldots,s_m,t_1,\ldots,t_m]
\]
be the coefficient of $x^{p^n-1}$ in $G(s)G_n(s,t)$, and let
\[
M_n(s,t)\in A[s_1,\ldots,s_m,t_1,\ldots,t_m]
\]
be the coefficient of $x^{p^n-1}$ in $G(t)G_n(s,0)$.
We note that
\[
G(s)G_n(s,t) \equiv H_n(s,t)x^{p^n-1}, \quad G(t)G_n(s,0) \equiv M_n(s,t) x^{p^n-1} \pmod{(x_0^{p^n},\ldots,x_N^{p^n})}.
\]

\item
For $n\geq 1$ and $1\le i\le m$, set
\[
L_{n,i}(s,t):=\frac{\partial H_n(s,t)}{\partial s_i}.
\]

\item
Finally, for $n\geq 1$, we define the $n\times m$ matrix
\[
\mathcal{L}_n(s,t):=
\begin{pmatrix}
L_{1,1}(s,t) & L_{1,2}(s,t) & \cdots & L_{1,m}(s,t) \\
L_{2,1}(s,t) & L_{2,2}(s,t) & \cdots & L_{2,m}(s,t) \\
\vdots       & \vdots       &        & \vdots       \\
L_{n,1}(s,t) & L_{n,2}(s,t) & \cdots & L_{n,m}(s,t)
\end{pmatrix}.
\]
We note that $\mathcal{L}_n(s,0)$ is the Jacobian matrix of
$H_1(s,0),\ldots, H_n(s,0)$ with respect to $(s_1,\ldots,s_m)$.
\item For $n\geq 1$,
we set 
\[
Z_n:=V(H_1(s,0)) \cap \cdots \cap V(H_n(s,0)) \cap V(I_{n}(\cL_n(s,0))) \subseteq \Spec{A[s_1,\ldots,s_m]},
\]
where $I_n(\cL_n(s,0))$ is the ideal generated by the $n \times n$ minors of $\cL_n(s,0)$.
Furthermore, for $b \in k^m$, we set
\[
\tau(b)=\min\{n \in \Z_{\geq 1} \mid b \in Z_n\}
\]
if the set is nonempty, and $\tau(b)=\infty$ otherwise.
\item For $n\geq 1$, the coefficient of
\[
x^{p^n-1-\alpha_i}:=x_0^{p^n-1-a_{i,0}}\cdots x_N^{p^n-1-a_{i,N}}
\]
in $G_n(s,t)$ is denoted by $G_{n,i}(s,t)$.
We then define the $n\times m$ matrix
\[
\mathcal{G}_n(s,t):=
\begin{pmatrix}
G_{1,1}(s,t) & G_{1,2}(s,t) & \cdots & G_{1,m}(s,t) \\
G_{2,1}(s,t) & G_{2,2}(s,t) & \cdots & G_{2,m}(s,t) \\
\vdots       & \vdots       &        & \vdots       \\
G_{n,1}(s,t) & G_{n,2}(s,t) & \cdots & G_{n,m}(s,t)
\end{pmatrix}.
\]
Note that the coefficient of $(x_0 \cdots x_N)^{p^n-1}$ in $G_n(s,t) \cdot G(c)$ is the $n$-th row of
\[
\mathcal{G}_n(s,t) 
\begin{pmatrix}
c_1 \\
c_2 \\
\vdots \\
c_{m}
\end{pmatrix}
\]
for any $c = (c_1, \ldots, c_m) \in k^m$.
\end{itemize}
\end{notation}

Note that the variables \(t_1,\ldots,t_m\), which correspond to the first-order coefficients, are not particularly essential in this section, but are essential for discussing the non-splitting index in mixed characteristic (cf.\ Section \ref{section:mixed}).

\begin{theorem}\label{rank-chara-ns}
Let $b=(b_1,\ldots,b_m) \in k^m$.
\begin{enumerate}
\item 
We
have
\[
\ns(G(b))\geq \min\{n \geq 1 \mid \rank(\mathcal{G}_n(b,0)) \leq n-1 \}.
\]
\item 
Let $n$ be a positive integer.
If $\sht(A/G(b)) \geq n$ and $\ns(G(b)) \geq n+1$, then $\rank(\cG_n(b,0))=n$.
\end{enumerate}
In particular, if $A/G(b)$ is not quasi-$F$-split, then
\[
\ns(G(b))= \min\{n \geq 1 \mid \rank(\mathcal{G}_n(b,0)) \leq n-1 \}
\]
\end{theorem}

\begin{proof}
We show (1). Let $n$ be the right-hand side of the desired inequality, which is finite since $\mathcal{G}_n (b,0)$ is an $n \times m$ matrix.
Since the case where $n=1$ is trivial, we may assume that $n\geq 2$.
We have $\rank(\cG_i(b,0))=i$ for $1 \leq i \leq n-1$ in this case.
Then there exists $c \in k^m$ such that
\begin{equation}\label{eq:G_i-rank}
\begin{pmatrix}
0 \\
0 \\
\vdots \\
0 \\
1
\end{pmatrix}
=
\mathcal{G}_{i}(b,0)
\begin{pmatrix}
c_1 \\
c_2 \\
\vdots \\
c_{m}
\end{pmatrix}. 
\end{equation}
We set $\alpha_1:=G(c)G(b)^{p-2}$.
Then $\alpha_1$ is homogeneous of degree $d(p-1)$.
If $n\ge 2$, then $\alpha_1\in \m^{[p]}$, and hence $\alpha_1\in \Ker(u)$.
We set 
\[
\alpha_2:=\theta_{G(b)}(F_*\alpha_1)=u(F_*(G_2(b,0)G(c))).
\]
Here, $\theta_{G(b)}$ is as in Convention~\ref{conv:local}.
Then the degree of $\alpha_2$ is $d(p-1)$ and $\alpha_2 \in \m^{[p]}$ by \eqref{eq:G_i-rank} if $3 \leq {i}$.
Thus, we have $\alpha_2 \in \Ker(u)$.
Therefore, if we put $\alpha_{j+1}=\theta_{G(b)}(F_*\alpha_j)$, then we have $\alpha_1,\ldots,\alpha_{i-1} \in \Ker(u)$ and $\alpha_i \notin \m^{[p]}$.
Thus we have $I^{\ns}_i(G(b)) \not\subseteq \m^{[p]}$ {for $1 \leq i \leq n-1$}.
{By \cref{Fedder-ns}, this implies $\ns(G(b)) \geq n$.}
{This completes the proof of (1).}

Next, we show (2) by induction on $n$.
For $n=1$, we have
\[
I^{\ns}_1(G(b))=G(b)^{p-2}A \not\subseteq \m^{[p]}
\]
by \cref{Fedder-ns}.
Thus, there exists a homogeneous element $\alpha \in A$ of degree $d$ such that $u(F_*G(b)^{p-2}\alpha) \neq 0$.
Taking  $c= (c_1, \ldots, c_m)\in k^m$ with $\alpha=G(c)$, we have
\[
\mathcal{G}_{1}(b,0) \begin{pmatrix}
c_1 \\
c_2 \\
\vdots \\
c_{m}
\end{pmatrix} \neq 0.
\]
Thus, we have $\rank(\mathcal{G}_1(b,0))=1$.

Next, we assume $n \geq 2$.
We have 
\[
I^{\ns}_1(G(b)),\ldots,I^{\ns}_n(G(b)) \not\subseteq \m^{[p]}
\]
by \cref{Fedder-ns}.
Then there exist $\alpha_1,\ldots,\alpha_{n-1} \in \Ker(u)$, $\alpha_n \in A$, and $\beta_1,\ldots,\beta_{n-1} \in A$ such that $\alpha_1 \in G(b)^{p-2}A$ is homogeneous of degree $d(p-1)$, $\beta_1,\ldots,\beta_{n-1}$ are homogeneous of degree $0$,
\[
\alpha_{i+1}=\theta_{G(b)}(F_*\alpha_i)+G(b)^{p-1}\beta_i
\]
for $1 \leq i \leq n-1$, and $\alpha_n \notin \m^{[p]}$.
In particular, there exists a homogeneous element $\alpha \in A$ of degree $d$ such that $\alpha_1=G(b)^{p-2}\alpha$.
Then we have
\[
\alpha_i=u^{i-1}(F^{i-1}_*(G_i(b)\alpha))+u^{i-2}(F^{i-2}_*(G_{i-1}(b,0)G(b)\beta_1))+\cdots+G_1(b,0)G(b)\beta_{i-1}
\]
for every $1 \leq i \leq n$.
Since $\beta_1,\ldots,\beta_{i-1}$ are homogeneous of degree $0$ and $A/G(b)$ is not $(n-1)$-quasi-$F$-split, we have
\[
u^{i-2}(F^{i-2}_*(G_{i-1}(b,0)G(b)\beta_1))+\cdots+G_1(b,0)G(b)\beta_{i-1} \in \Ker(u) \cap \m^{[p]}
\]
by \Cref{variant-fedder}.
Therefore, we may assume $\beta_1=\cdots=\beta_{n-1}=0$.
In particular, we have
\[
\alpha_i=u^{i-1}(F^{i-1}_*(G_i(b,0)\alpha)) \in \Ker(u)
\]
for $1 \leq i \leq n-1$.
Since $\alpha_i$ is homogeneous of degree $d(p-1)$, the coefficient of $x^{p^i-1}$ in $G_i(b,0)\alpha$ is zero.
On the other hand, we have
\[
\alpha_n=u^{n-1}(F^{n-1}_*(G_n(b,0)\alpha)) \notin \m^{[p]}. 
\]
In particular, the coefficient of $x^{p^n-1}$ in $G_n(b,0)\alpha$ is nonzero; denote it by $\lambda$.
Then $\lambda\neq 0$.
Taking  $c= (c_1, \ldots, c_m)\in k^m$ with $\alpha=G(c)$, we have
\[
\begin{pmatrix}
0 \\
0 \\
\vdots \\
0 \\
\lambda
\end{pmatrix}
=
\mathcal{G}_{n}(b,0)
\begin{pmatrix}
c_1 \\
c_2 \\
\vdots \\
c_{m}
\end{pmatrix}.
\]
Therefore, we obtain 
\[
\rank(\mathcal{G}_n(b,0)) \geq \rank(\mathcal{G}_{n-1}(b,0))+1.
\]
Hence $\rank(\cG_n(b,0))=n$.
\end{proof}

\begin{corollary}\label{ns-vs-qfs}
Let $b=(b_1,\ldots,b_m) \in k^m$.
Then $A/(G(b))$ is not quasi-$F$-split if and only if $\ns(G(b))<\infty$.
\end{corollary}

\begin{proof}
The ``if'' direction follows from \cref{finite-ns-to-non-qfs}.
We prove the ``only if'' direction.
By \cref{rank-chara-ns} (2), it suffices to show that
\[
\min\{n \geq 1 \mid \rank(\mathcal{G}_n(b,0)) \leq n-1 \}<\infty.
\]
However,
\[
\min\{n \geq 1 \mid \rank(\mathcal{G}_n(b,0)) \leq n-1 \} \leq m+1.
\]
since $\mathcal{G}_{n} (b,0)$ is an $n \times m$ matrix.
This proves the claim.
\end{proof}

\begin{lemma}\label{derivation-delta}
We have
\[
\frac{\partial \Delta(G)}{\partial s_i}(s,t)
=
G(s)^{p-1}x^{\alpha_i}-s_i^{p-1}x^{p\alpha_i}.
\]
\end{lemma}

\begin{proof}
By definition, we have
\begin{align*}
\frac{\partial}{\partial s_i}
\Bigl(
\frac{1}{p}\bigl(\wt{G}(s,t)^p-\phi(\wt{G}(s,t))\bigr)
\Bigr)
&=
\frac{1}{p}
\frac{\partial}{\partial s_i}
\bigl(\wt{G}(s,t)^p-\phi(\wt{G}(s,t))\bigr) \\
&=
\wt{G}(s,t)^{p-1}x^{\alpha_i}
-
s_i^{p-1}x^{p\alpha_i}.
\end{align*}
Taking the image under the natural reduction map
\[
\wt{A}[s_1,\ldots,s_m,t_1,\ldots,t_m]
\rightarrow
A[s_1,\ldots,s_m,t_1,\ldots,t_m],
\]
we obtain the desired formula.
\end{proof}

\begin{proposition}\label{relation-M-L}
For every $1\le i\le m$ and $n \geq 1$, we have
\[
G_{n,i}(s,t) \equiv -L_{n,i}(s,t) \pmod{(H_1(s,t),H_{n-1}(s,t))},
\]
where $H_0(s,t)=0$.
\end{proposition}

\begin{proof}
Fix $1\le i\le m$ and $n \geq 1$.
We note that the coefficient of $x^{p^n-1}$ in
\[
\frac{\partial}{\partial s_i}\bigl(G_n(s,t)G(s)\bigr)
\]
is equal to $L_{n,i}(s,t)$ by definition.

Since $G_1(s,t)G(s)=G(s)^{p-1}$ and
$\frac{\partial G(s)}{\partial s_i}=x^{\alpha_i}$,
the assertion is clear for $n=1$.

Now assume $n\ge 2$.
We have
\begin{align}\label{eq:partial}
\frac{\partial}{\partial s_i}\bigl(G_n(s,t)G(s)\bigr)
&=
\frac{\partial}{\partial s_i}
\bigl(G_{n-1}(s,t)^p G(s)^{p-1}\Delta(G)(s,t)\bigr) \notag \\
&=G_{n-1}(s,t)^p
\Bigl(
-(p-1)G(s)^{p-2}\frac{\partial G(s)}{\partial s_i}\Delta(G)(s,t)
+G(s)^{p-1}\frac{\partial\Delta(G)(s,t)}{\partial s_i}
\Bigr) \notag \\
&=-G_n(s,t)x^{\alpha_i}
+G_{n-1}(s,t)^pG(s)^{p-1}
\frac{\partial\Delta(G)}{\partial s_i}(s,t).
\end{align}

Moreover, it follows from \cref{derivation-delta} that
\begin{align*}
&G_{n-1}(s,t)^pG(s)^{p-1}
\frac{\partial\Delta(G)}{\partial s_i}(s,t) \\
& =
G_{n-1}(s,t)^pG(s)^{p-1}
\bigl(G(s)^{p-1}x^{\alpha_i}-s_i^{p-1}x^{p\alpha_i}\bigr) \\
& =
(G_{n-1}(s,t)G(s))^p G(s)^{p-2}x^{\alpha_i}
-(G_{n-1}(s,t)x^{\alpha_i})^p G(s)^{p-1}s_i^{p-1} \\
& \equiv
H_{n-1}(s,t)^p G(s)^{p-2}x^{p^n-p+\alpha_i}
-G_{n-1,i}(s,t)^p H_1(s,t)s_i^{p-1} x^{p^n-1}
\pmod{(x_0^{p^n},\ldots,x_N^{p^n})}.
\end{align*}
Therefore, the coefficient of $x^{p^n-1}$ in the second term of
\eqref{eq:partial} is contained in the ideal
$(H_1(s,t),H_{n-1}(s,t))$.

Consequently, we obtain
\[
\frac{\partial}{\partial s_i}\bigl(G_n(s,t)G(s)\bigr)
\equiv
-G_n(s,t)x^{\alpha_i}
\pmod{(x_0^{p^n},\ldots,x_N^{p^n},H_1(s,t),H_{n-1}(s,t))}.
\]
Note that the both-hand sides are homogeneous of degree $(p^n-1)d$ with respect to $x_0, \ldots, x_N$.
Comparing the coefficients of $x^{p^n-1}$, it follows that the coefficient
of $x^{p^n-1-\alpha_i}$ in $G_n(s,t)$ is congruent to $-L_{n,i}(s,t)$
modulo $(H_1(s,t),H_{n-1}(s,t))$.
\end{proof}

\begin{theorem}\label{tau-vs-ns}
For \(b\in k^m\), we have
\[
\tau(b)=\ns(G(b)).
\]
\end{theorem}

\begin{proof}
We first show that \(\ns(G(b))\le \tau(b)\).
Assume that \(b\in Z_n\) for some integer \(n\ge 1\).
Then
\[
\rank(\cL_n(b,0))\le n-1
\]
by the definition of \(Z_n\).
Since \(b\in V(H_i)\) for all \(1\le i\le n\), Proposition~\ref{relation-M-L} yields
\[
\rank(\cG_n(b,0))\le n-1.
\]
Since $A/G(b)$ is not $n$-quasi-$F$-split by \cite{kty}*{Theorem~C}, we have
\[
\ns(G(b))\le n
\]
by \cref{rank-chara-ns}. Therefore
\[
\ns(G(b))\le \tau(b).
\]

Next, we prove the converse inequality.
We may assume $\ns(G(b))<\infty$.
Then we have
\[
n:=\min\{r\ge 1\mid \rank(\cG_r(b,0))\le r-1\} \leq \ns(G(b))
\]
by \cref{rank-chara-ns} and $A/G(b)$ is not quasi-$F$-split by \cref{ns-vs-qfs}.
In particular, we have
\begin{equation}\label{eq:b-in-intersection}
b\in \bigcap_{i\geq 1}V(H_i(s,0)).
\end{equation}
By Proposition~\ref{relation-M-L} and \eqref{eq:b-in-intersection}, we obtain
\[
\rank(\cL_{n}(b,0))\le n-1.
\]
This shows that \(b\in Z_{n}\).
Therefore
\[
\tau(b)\le n\le \ns(G(b)).
\]
This completes the proof.
\end{proof}

\subsection{Formula for the non-splitting indices}\label{section:explicit-formula}

Let \(k\) be an algebraically closed field of characteristic \(p>0\).
Let
\[
A:=k[x_0,\ldots,x_N]
\]
be a weighted polynomial ring with \(\deg(x_i)=q_i\in \Z_{>0}\), and set
\[
d:=q_0+\cdots+q_N
\]
as in Notation \ref{notation:AI-vs-NS}.
Let \(V\) be the \(k\)-vector space of homogeneous polynomials of degree \(d\), and let \(f\in V\).
Set $X:=\Proj(A/fA)$.
Let $M_1,\ldots,M_m$ be all degree $d$ monomials (with coefficients 1) in $A$. 

Let
\[
f=\sum_{i=1}^m b_i M_i\in V
\]
be the monomial decomposition of \(f\).
We define
\[
\Delta(f)
:=
\sum_{\substack{0\le \beta_1,\dots,\beta_m\le p-1\\
\beta_1+\cdots+\beta_m=p}}
\frac{1}{p}\binom{p}{\beta_1,\dots,\beta_m}
(b_1M_1)^{\beta_1}\cdots(b_mM_m)^{\beta_m}.
\]
Set 
\[
v_f:=
\begin{pmatrix}
    b_1 \\
    \vdots \\
    b_m
\end{pmatrix}.
\]

Consider the basis
\[
\{F_*(x_0^{i_0}\cdots x_N^{i_N}) \mid 0\le i_0,\ldots,i_N\le p-1\}
\]
of \(F_*A\).
Let
\[
u\colon F_*A\to A
\]
be the dual basis element corresponding to
\[
F_*\bigl((x_0\cdots x_N)^{p-1}\bigr).
\]

We define the $m \times m$ matrix $T$ by the representation matrix of
\[
F_*V\to V,\qquad
h\longmapsto u\!\left(F_*\bigl(\Delta(f)f^{p-2}h\bigr)\right)
\]
with respect to $F_*M_1,\ldots,F_*M_m$ and $M_1,\ldots,M_m$.
Furthermore, we define the vector $\lambda$ by
\[
\lambda:=
\begin{pmatrix}
    u\!\left(F_*\bigl(f^{p-2}M_1\bigr)\right) & \cdots & u\!\left(F_*\bigl(f^{p-2}M_m\bigr)\right) 
\end{pmatrix}
\]
Here \(\lambda\) is \(k\)-valued since \(u(F_*(f^{p-2}M_i))\) is homogeneous of degree \(0\) for every $i$.
Define row vectors \(R_n\in k^{1\times m}\) recursively by
\[
R_1:=F(\lambda),\qquad
R_{n+1}:=F(R_n)T\qquad (n\ge1).
\]

\begin{theorem}\label{explicit-formula}
In the above setting, the quasi-$F$-split height $\sht (X)$ satisfies
\[
\sht(X)=\inf\{\,n\ge 1\mid R_n v_f \neq 0\,\},
\]
where the infimum is defined to be \(\infty\) if the set on the right-hand side is empty.

Furthermore, if $R_nv_f=0$ for every $n \geq 1$, then
\[
\ns(f)
=
\inf\left\{\,n\ge 1\ \middle|\ 
\operatorname{rank}
\begin{pmatrix}
R_1\\
R_2\\
\vdots\\
R_n
\end{pmatrix}
\le n-1
\right\}.
\]
\end{theorem}

\begin{proof}
For \(h\in V\), let \(v_h\) be the column vector determined by
\[
h=(M_1,\dots,M_m)v_h.
\]
Then
\[
(M_1 \ldots M_m) T \circ F^{-1}(v_h)=u(F_*(\Delta(f)f^{p-2}h)).
\]
Furthermore, we have
\[
\lambda F^{-1}(v_h)=u(F_*(f^{p-2}h)).
\]
Therefore, we have
\[
\lambda F^{-1}((T \circ F^{-1})^{n-1})(v_h)=u^n(F^n_*((f^{p(p-2)}\Delta(f)))^{1+p+\cdots+p^{n-2}}f^{p-2}h).
\]
In particular, we have
\begin{equation}\label{eq:R_n}
R_nv_h=F^n(\lambda F^{-1}((T \circ F^{-1})^{n-1})(v_h))=F^n \circ u^n(F^n_*((f^{p(p-2)}\Delta(f)))^{1+p+\cdots+p^{n-2}}f^{p-2}h).
\end{equation}
We have
\begin{align*}
   R_nv_f=F^n \circ u^n(F^n_*((f^{p(p-2)}\Delta(f)))^{1+p+\cdots+p^{n-2}}f^{p-1}).
\end{align*}
Thus, by \Cref{variant-fedder}, we have
\[
\sht(X)=\inf\{\,n\ge 1\mid R_nv_f\neq 0\,\}.
\]

Next, we prove the second assertion.
We assume $\sht(X)=\infty$, that is, $X$ is not quasi-$F$-split.
We use notations in \cref{notation:AI-vs-NS}.
By reordering appropriately, we assume that $M_i = x^{\alpha_i}$.
Then we have \(G(b)=f\).
Furthermore, we have
\[
R_l v_i =F^l \circ u^l(F^l_*(G_l(b,0)M_i)),
\]
where $v_i$ is the $i$-th standard basis vector of $k^m$.
Since the right-hand side is $G_{l,i}(b,0)$, we obtain
\[
\begin{pmatrix}
R_1\\
R_2\\
\vdots\\
R_{n}
\end{pmatrix}
=\mathcal{G}_n(b,0).
\]
Thus, the result follows from \Cref{rank-chara-ns}.
\end{proof}

\section{Ogus's result on K3 surfaces}
\label{Section:Ogus}

\subsection{Moduli stack of K3 surfaces}

\begin{definition}
Let $k$ be an algebraically closed field.
A smooth projective surface $X$ over $k$ is called a \emph{K3 surface} if $\Omega_{X/k}^2 \simeq \cO_{X}$ and $H^1(\cO_{X}) =0$.
A \emph{polarization} $L$ on $X$ is an ample line bundle on $X$.
A polarization $L$ is called \emph{primitive} if $L$ is not a positive power of a line bundle on $X$.
The \emph{degree} of $L$ is the self-intersection number $(L,L) \in 2\Z.$
The Artin-Mazur height of $X$ is denoted by $\sht (X)$.
We say $X$ is \emph{supersingular} when $\sht (X) = \infty.$
\end{definition}

Note that, by \cite{Yobuko}, $\sht (X)$ is equal to the quasi-$F$-split height of $X$.
Also, we note that $\sht (X) \in \{1,2, \ldots, 10, \infty\}$ (see \cite{Huybrechts}*{Chapter 18, Section 3} for example).

\begin{definition}
Let $S$ be a scheme.
A \emph{K3 space} $X$ over $S$ is a smooth proper algebraic space whose any geometric fiber is a K3 surface.
A primitive polarization $L$ of degree $2D$ on $X$ is an element $L \in \Pic_{X/S} (S)$ that restricts to a primitive polarization of degree $2D$ on each geometric fiber.
\end{definition}

\begin{definition}
Let $D\in\Z$.
Let $\mathcal{F}_{2D, \Z}$ be the moduli functor that sends a scheme $S$ to the groupoid consisting of primitively polarized K3 spaces $(X,L)$ over $S$ of degree $2D$.
\end{definition}

The stack $\mathcal{F}_{2D, \Z}$ is a Deligne-Mumford stack of finite type over $\Z$ by \cite{Rizov}*{Theorem 4.3.4} or \cite{Maulik}*{Proposition 2.1}.
We fix an algebraically closed field $k$ of characteristic $p>0.$
We denote the base change of moduli stack by $\mathcal{F}_{2D, k}$, or simply $\mathcal{F}_{2D}$.

\begin{definition}
Let $h \in \{1, 2, \ldots, 10, 11\}$.
We let 
\[
\mathcal{F}_{2D, 11} \subset
\mathcal{F}_{2D, 10} \subset \cdots \subset
\mathcal{F}_{2D, 1} = \mathcal{F}_{2D}
\]
be the moduli substacks of $\mathcal{F}_{2D}$ defined in \cite[p.334]{OgusK3}.
In particular, $\mathcal{F}_{2D,h} (k)$ consists of primitively polarized K3 surfaces $(X,L)$ over $k$ with $\sht (X) \geq h$.
\end{definition}

\begin{remark}
The substack $\mathcal{F}_{2D, h}$ is not necessarily reduced a priori. See \cref{thm:Ogus}.
\end{remark}

To state Ogus's theorem, we recall the definition of the Artin invariant and its variant.

\begin{definition}
\label{defn:artin-invariant}
Let $X$ be a K3 surface over an algebraically closed field $k$ of characteristic $p>0$.
\begin{enumerate}
    \item 
    When $X$ is supersingular,
    the Artin invariant of $X$ is the integer $\sigma (X)$ such that
    \[
    \disc \NS (X) = - p^{2\sigma (X)}.
     \]
     We have $\sigma (X) \in \{1,2, \ldots, 10\}$ (see \cite{Huybrechts}*{Chapter 17, Section 2} for example).
     We set $\sigma (X) := \infty$ for non-supersingular K3 surfaces.
    \item
    Let $L$ be a primitive polarization of degree $2D$ on $X$.
    Let $L_{\dR} \in H^2_{\dR} (X/k)$ be the first Chern class of $L$.
   When $X$ is supersingular, we define $\tau (X,L) \in \Z$ by
   \[
   \tau (X,L) = 
   \begin{cases}
   \sigma(X) &  \textup{if }L_{\dR} \notin E_{\sigma(X)}, \\
   \sigma(X) -1 & \textup{if } L_{\dR} \in E_{\sigma(X)}.
   \end{cases}
   \]
   For the definition of $E_{\sigma(X)}$, see \cite[p.328]{OgusK3} (see also \cite[Lemma 5]{OgusK3} for other descriptions).
   We set $\tau (X,L) = \infty$ for non-supersingular K3 surfaces.
\end{enumerate}
\end{definition}

\begin{remark}
\label{remark:tau}
\begin{enumerate}
    \item
Ogus \cite{OgusK3} treats the more general notion of
$p$-primitive polarizations, but for simplicity, we restrict ourselves to primitive polarizations.    
\item
Let $Q \in \Sym^2 (H^2_{\mathrm{cris}} (X/W(k))^{\vee})$ be the quadratic form given by
\[
Q(x) = \frac{1}{2} (x,x),
\]
where $(,)$ denotes the cup product on $H^2_{\mathrm{cris}} (X/W(k))$.
We denote the quadratic form on $H^2_{\dR} (X/k)$ induced by $Q \mod p$ by the same notation $Q$.
Then by the definition, $E_{\sigma (X)}$ is totally isotropic with respect to $Q$.
In particular, if $p \nmid D$, then we have $L_{\dR} \notin E_{\sigma(X)}$ and $\tau (X,L) = \sigma (X)$. 
\item
By \cite[Lemma 5]{OgusK3} and the Tate conjecture for K3 surface (\cite{Madapusi}, \cite{Kim-Madapusi}, and \cite{Madapusierratum} see also \cite{Ito-Ito-Koshikawa}*{Remarks 6.9 and 6.10}) (cf.\ \cite[Corollary 1.5]{Ogus-supersingularK3crystal}),
$E_{\sigma (X)}$ is isomorphic to 
\[
\Im (\mathrm{ch}_{\dR} \colon\Pic (X) \otimes_{\Z} k \rightarrow H^2_{\dR} (X/k)) ^{\perp} \subset H^2_{\dR} (X/k).
\]
In particular, if there exists a curve $V \subset X$ such that $(L,V) =1$ (i.e.\ $V$ is a line), then we have $\tau (X, L) = \sigma (X)$.
\end{enumerate}
\end{remark}

Ogus proved the following results.

\begin{theorem}[{\cite{OgusK3}*{Theorem 15}}, cf.\ {\cite{vdGeer-Katsura}}]
\label{thm:Ogus}
Let $h \in \{1,2,\ldots, 10,11\}.$
\begin{enumerate}
\item 
The closed substack $\mathcal{F}_{2D,h} \subset \mathcal{F}_{2D}$ is a locally complete intersection of purely dimension $20-h$.
\item
Let $(X,L) \in \mathcal{F}_{2D,h} (k)$ be a primitively polarized K3 surface of degree $2D$ over $k$ with $\sht(X) \geq h$.
Then the following are equivalent.
\begin{enumerate}
    \item 
The stack $\mathcal{F}_{2D, h}$ is singular at $(X,L)$.
\item 
The K3 surface $X$ is supersingular and $\tau(X,L) <h.$
\end{enumerate}
\end{enumerate}
In particular, $\mathcal{F}_{2D,h}$ is reduced if $1 \leq h \leq 10$, and generically non-reduced if $h=11.$
\end{theorem}

\subsection{Families of smooth K3 hypersurfaces}

As in the previous section, let \(k\) be a fixed algebraically closed field of characteristic \(p>0\).
In this subsection, we let $A:=k[x_0, \ldots x_3]$ be a polynomial ring with graded structure with 
\[
(\deg (x_0), \ldots \deg(x_3)) = \quad \mathrm{either}\quad (1,1,1,1) \quad \textup{or} \quad (1,1,1,3).
\]
We put $d :=4$ and $D :=2$ (resp.\ $d:=6$ and $D :=1$) in the former (resp.\ latter) case.
Let
\[
W:=\Proj A,\qquad S_d:=A_d=H^0(W,\mathcal O_W(d)).
\]
Then we have a universal hypersurface
\[
\mathcal X \subset \P(S_d)\times W
\]
over $\P(S_d)$.
Let $B_{2D}\subset \P(S_d)$ be the open locus over which
$\mathcal X\to \P(S_d)$ is smooth.
Then the base change $\pi \colon\mathcal{X}_{B_{2D}} \rightarrow B_{2D}$ is a smooth K3 space.

\begin{remark}
\begin{enumerate}
\item 
By the explicit construction, one can show that the above $B_{2D}$ is non-empty.
The case where $D=2$ is well-known.
In the case where $D=1$, we can show that
$f=x_0^6+x_1^6+x_2^6+x_3^2$ (resp.\ $f=x_3^2+x_3x_1^3+x_0^5x_2 + x_1x_2^5
$) defines a smooth K3 surface when $p\neq 2,3$ (resp.\ $p\neq 5$).

\item
There exist 95 families of (well-formed) RDP K3 hypersurfaces in weighted projective spaces.
However, among them, only the above two degrees give smooth members, i.e.\ $B_{2D} \neq \emptyset$ (see \cite{Fletcher}*{II.3.3}).
\end{enumerate}
\end{remark}

We shall relate this $B_{2D}$ with the moduli stack of polarized K3 surfaces.
For every geometric point $b\in B_{2D}$, the fiber $X_b$ is a smooth hypersurface of degree $d$ in $W$.
In both cases, the restriction
$
L_b:=\mathcal O_W(1)|_{X_b}$
is an ample line bundle on $X_b$.
In the quartic case, this is obvious. In the weighted sextic case, $W=\P(1,1,1,3)$ has
a unique singular point, and every smooth member avoids it; hence $\mathcal O_W(1)|_{X_b}$ is
again a line bundle.
Moreover, we have
$L_b^2=4$ (resp.\ $L_b^2=2$) in the former case (resp.\ in the latter case).
By the degree reason, this is a primitive polarization.
Therefore the family $(\mathcal X_{B_{2D}},\mathcal L)$, where
$
\mathcal L:=\mathcal O_W(1)|_{\mathcal X_{B_{2D}}},
$
defines a morphism
$\mu:B_{2D}\rightarrow \mathcal F_{2D}.$

\begin{proposition}
\label{prop:smoothsurjection}
The morphism $\mu$ is smooth and representable. 
\end{proposition}

\begin{proof}


\smallskip
\noindent
{\bf Case 1: $(\deg x_i)=(1,1,1,1)$ and $D=2$.}

Let $\mathcal U_4\subset \mathcal F_{4}$ be the substack whose $S$-points are polarized K3
spaces $(\pi:X\to S,\lambda)$ such that for every geometric point $\overline s\to S$, the polarization
$\lambda_{\overline s}$ is a very ample line bundle.
Clearly, this is an open substack.

By construction, for every $b\in B_{2D}$, the polarized K3 surface
$(X_b,\mathcal O_{X_b}(1))$
lies in $\mathcal U_4$, so $\mu$ factors through $\mu' \colon B_4 \rightarrow \mathcal U_4$.

We show that $\mu'$ admits a section \'{e}tale locally (, and in particular, $\mu'$ is surjective.)
 let $(f:X\to S,L)$ be an $S$-point of $\mathcal U_4$.
After replacing $S$ by an \'etale cover, we may assume that $f_*L$ is free of rank $4$.
Choosing a basis gives an identification
$\P(f_*L)\simeq \P^3_S,
$
and the complete linear system $|L|$ yields a closed immersion $X\hookrightarrow \P^3_S.
$
Fiberwise, the image is a quartic surface. Hence its ideal sheaf $\mathcal I_X$
satisfies
$
\mathcal I_{X_{\overline s}}\simeq \mathcal O_{\P^3_{\overline s}}(-4)
$
for every geometric point $\overline s\to S$.
By the see-saw theorem, there exists a line bundle $M$ on $S$ such that
$
\mathcal I_X(4)\simeq p^*M.
$
Equivalently, $X$ is cut out by a section of
$p^*M^{-1}\otimes \mathcal O_{\P^3_S}(4).$
After replacing $S$ by an \'etale cover, we may assume that $M$ is trivial, and then
$X$ is cut out by a section of $\mathcal O_{\P^3_S}(4)$.
Thus, after an \'etale localization on $S$, the family is induced from the universal quartic family,
so $\mu:B_{4}\to \mathcal U_4$ is surjective.

To prove the smoothness, one argues exactly as in Step~3 of the proof of \cite{Rizov}*
{Theorem~4.3.3}, with $L$ in place of $L^{\otimes 3}$.
Namely, for any morphism $S\to \mathcal U_4$ corresponding to $(X/S,L)$, after an \'etale cover
$S'\to S$,
$S'\times_{\mathcal U_4} B$
is isomorphic to 
\[
\Aut_{\P^3_{S'}/S'} \simeq \PGL_{4, S'}.
\]
In particular, it is smooth over $S'$, and therefore $B_{4}\to \mathcal U_4$ is smooth.
The representability also follows from the same argument as in \cite{Rizov}*{Theorem~4.3.3}.

\smallskip
\noindent
{\bf Case 2: $(\deg x_i)=(1,1,1,3)$ and $D=1$.}

Note that we have
\[
h^0 (\lambda_{\overline{s}}) =3, \quad h^1 (\lambda_{\overline{s}}) = h^2 (\lambda_{\overline{s}}) = 0
\]
for any geometric point $\overline{s}$ on a $k$-scheme $S$ and any $S$-point $(\pi \colon X \rightarrow S, \lambda)$ of $\mathcal{F}_2$.
Let $\mathcal{U}_2 \subset \mathcal{F}_2$ be the substack whose $S$-points are $(\pi \colon X \rightarrow S, \lambda)$ such that for every geometric point $\overline{s} \rightarrow S,$ the polarization $\lambda_{\overline{s}}$ is a base-point free line bundle.
Clearly, this is an open substack.

We claim that $\mu$ factors through $\mathcal U_2$.
Since $\mathcal U_2\hookrightarrow \mathcal F_{2}$ is an open immersion, it suffices to check this on geometric points.
Let $\overline b\to B_{2}$ be a geometric point.
Then $\mathcal{X}_{\overline{b}}\subset \P(1,1,1,3)$ is a smooth weighted sextic, so its defining equation has the form
\[
x_3^2+x_3\,q_3(x_0,x_1,x_2)+q_6(x_0,x_1,x_2)=0.
\]
The complete linear system $|\mathcal{L}_{\overline{b}}|$ is generated by the weight-$1$ coordinates $x_0, x_1, x_2$, hence induces a morphism
\[
\varphi_{|\mathcal{L}_{\overline b}|}:\mathcal{X}_{\overline b}\to \P^2_{\overline b}
\]
(, which is a finite flat morphism of degree $2$).
Therefore $\mu$ factors through
$
\mu'\colon B_{2}\rightarrow \mathcal U_2.
$

Next, we prove that $\mu' \colon B_{2} \rightarrow \mathcal{U}_2$ admits a section \'{e}tale locally (and in particular, $\mu'$ is surjective.).
Let $S\to \mathcal U_2$ correspond to a polarized K3 space $(f:X\to S, \lambda)$.
After replacing $S$ by an \'etale cover, we may assume that $\lambda$ is a line bundle $L$, and $f_*L$ is free of rank $3$, so that $\P(f_*L)\simeq \P^2_S$.
Note that $f^* f_* L \rightarrow L$ is surjective by the definition of $\mathcal{U}_2$.
Let
\[
\varphi:=\varphi_L:X\to \P(f_*L)\simeq \P^2_S
\]
be the induced morphism.
By the fiberwise criterion for flatness, $\varphi$ is a finite flat double cover.
Set $\mathcal A:=\varphi_*\mathcal O_X.$
Note that $\mathcal A$ is a locally free
$\mathcal O_{\P^2_S}$-algebra of rank $2$.
The unit section of $\mathcal A$ induces a morphism
$
u:\mathcal O_{\P^2_S}\rightarrow \mathcal A.
$
Then $u$ defines a line subbundle of $\mathcal A$, and the cokernel
$
\mathcal Q:=\operatorname{coker}(u)$
is invertible.
We therefore obtain an exact sequence
\[
0\rightarrow \mathcal O_{\P^2_S}
\rightarrow \mathcal A
\rightarrow \mathcal Q
\rightarrow 0.
\]
We next determine $\mathcal Q$.
Since $\det(\mathcal A)\simeq \mathcal Q$, we have
$
\mathcal A^\vee\simeq \mathcal A\otimes \mathcal Q^{-1}.
$
By the Grothendieck duality,
\[
\varphi_*\omega_{X/S}
\simeq
\mathcal Hom_{\P^2_S}(\mathcal A,\omega_{\P^2_S/S})
\simeq
\mathcal A^\vee\otimes \omega_{\P^2_S/S}
\simeq
\mathcal A\otimes \mathcal Q^{-1}\otimes \omega_{\P^2_S/S}.
\]
After replacing $S$ by an \'{e}tale cover, we may assume that
$\omega_{X/S} \simeq \cO_{X}$.
Therefore, we have
$
\mathcal{A} \simeq \mathcal{A} \otimes \mathcal{Q}^{-1} \otimes \omega_{\P^2_{S}/S}.
$
Taking the determinant, we have $\mathcal{Q}^{\otimes 2} \simeq \omega_{\P^2_{S}/S}^{\otimes 2}$.
After replacing $S$ again, we may assume that $\mathcal{Q} \simeq \omega_{\P^2_{S}/S} \simeq \cO_{\P^2_{S}}(-3)$.
Thus
\[
0\rightarrow \mathcal O_{\P^2_S}
\rightarrow \mathcal A
\rightarrow \mathcal O_{\P^2_S}(-3)
\rightarrow 0.
\]
After replacing $S$ by an affine cover, this sequence splits.
Indeed, if $p:\P^2_S\to S$ denotes the projection, then
\[
\Ext^1_{\P^2_S}\bigl(\mathcal O_{\P^2_S}(-3),\mathcal O_{\P^2_S}\bigr)
=
H^1(\P^2_S,\mathcal O_{\P^2_S}(3))
=
H^1\bigl(S,p_*\mathcal O_{\P^2_S}(3)\bigr)
=0,
\]
because $R^1p_*\mathcal O_{\P^2_S}(3)=0$ and $S$ is affine.
Hence
$
\mathcal A\simeq
\mathcal O_{\P^2_S}\oplus \mathcal O_{\P^2_S}(-3).
$
Then we have
\[
\varphi_* L \simeq \cO_{\P^2_S} (1) \oplus \cO_{\P^2_{S}}(-2), \quad \varphi_{*} L^{\otimes3} \simeq \cO_{\P^2_S} (3) \oplus \cO_{\P^2_{S}}.  
\]
Take a $\cO_{S}$-basis $\{x_0, x_1, x_2\}$ of $H^0 (\varphi_* L) = H^0 (L)$.
Let $x_3 \in H^0 (\varphi_* L^{\otimes 3}) = H^0 (L^{\otimes 3})$ be the section corresponding to the unit of the second factor $\cO_{\P^2_S}$.
Then $(x_0,x_1,x_2,x_3)$ defines a morphism 
\[
h \colon X \rightarrow \P_S(1,1,1,3).
\]
The algebra structure of $\mathcal{A}$ gives a morphism
\[
\cO_{\P^2_S} (-3) \otimes \cO_{\P^2_S} (-3) \rightarrow \cO_{\P^2_S} \oplus \cO_{\P^2_S} (-3),
\]
corresponding to global sections $q_3 \in \cO_{\P^2_S} (3)$ and $q_6 \in \cO_{\P^2_S} (6)$.
Then we can show that $h$ induces an isomorphism 
\[
X \simeq \{x_3^2 = x_3q_3 (x_0, x_1, x_2)  + q_6 (x_0, x_1, x_2)\} \subset \P_S (1,1,1,3).
\]
Therefore, after \'etale localization on $S$, $X$ is isomorphic to $\mathcal{X}_S$ for some morphism $S\to B_2.$
Thus $\mu'$ admits a section \'{e}tale locally.

We now prove that
$\mu'$
is smooth.
Let $S \rightarrow \mathcal{U}_2$ be an $S$-point.
It suffices to show that $B_2 \times_{\mathcal{U}_2} S$ is smooth over $S$.
By what we have seen above, we may assume that there is a lift
$
\sigma:S\rightarrow B_{2}
$
of a given morphism
$
S\rightarrow \mathcal U_2$ such that
the corresponding weighted sextic family
\[
i: \mathcal{X}_S \hookrightarrow W_S:=\P_S(1,1,1,3)
\]
is defined by a global section of $H^0 (\sO_{W_S} (6))$.
Let $
\mathcal{L}_S:=i^*\mathcal O_{W_S}(1).$
We write
\[
G:=\Aut_{W_S/S}=\Aut_{\P_S(1,1,1,3)/S}.
\]
We show that
$
B_2\times_{\mathcal{U}_2}S \simeq G
$
as $S$-schemes.
We construct a morphism
\[
\theta:G\rightarrow B_2\times_{\mathcal{U}_2}S
\]
as follows.
For any $S$-scheme $T$ and any
\[
g \in G(T)=\Aut_T\bigl(W_T\bigr),
\]
the image
$
g (\mathcal{X}_T)\subset W_T 
$
(where $\mathcal{X}_T := \mathcal{X}_S \times _S T$)
is again a smooth weighted sextic, hence determines a morphism
$g \cdot \sigma_T:T\rightarrow  B_{2}.$
Moreover, the restriction
\[
g|_{\mathcal{X}_T}:\mathcal{X}_T\stackrel{\sim}{\rightarrow} g(\mathcal{X}_T)
\]
is an isomorphism preserving the polarization $\mathcal O(1)$.
This gives a $T$-point of
$B_2\times_{\mathcal{U}_2}S$.

We now construct the inverse map on $T$-points.
Let $T\to S$ be an $S$-scheme.
A $T$-point of
$
B_2\times_{\mathcal{U}_2}S
$
consists of a morphism
$
\beta:T\to B_{2}
$
corresponding to $(\mathcal{X}_{\beta}, L_{\beta})$,
together with an isomorphism 
$\alpha:\mathcal{X}_T\xrightarrow{\sim}\mathcal{X}_\beta$
such that $\alpha^{\ast} \mathcal{L}_{\beta} \simeq \mathcal{L}_T \otimes \pi_{T}^{*} \mathcal{M}$ for some line bundle $\mathcal{M}$ on $T$,
where
\[
i_T:\mathcal{X}_T\hookrightarrow W_T,\qquad i_\beta:\mathcal{X}_\beta\hookrightarrow W_T
\]
denote the two embeddings into $W_T=\P_T(1,1,1,3)$.
Let $T = \bigcup U_i$ be an open cover where $\mathcal{M}_{U_i}$ is trivial. 
We also assume that $U_i$ is affine.
We fix $i$ in the following.
We choose a linearization $\widetilde{\alpha} \colon \alpha^{\ast} (\mathcal{L}_{\beta})|_{U_i} \simeq (\mathcal{L}_T)|_{U_i}$, that is unique up to $\Gamma (\cO_{U_i}^{\times})$.

We claim that there exists a unique
$g\in G(T)$
such that
$
i_\beta\circ \alpha = g\circ i_T.
$
For each \(n\ge 0\), the chosen linearization \(\widetilde{\alpha}\) induces an isomorphism
\[
\widetilde{\alpha}^{\otimes n}:\alpha^{*}(\mathcal L_{\beta}^{\otimes n})|_{U_i}
\;\xrightarrow{\ \sim\ }\;
(\mathcal L_T^{\otimes n})|_{U_i}.
\]
Since \(\alpha\) is an isomorphism over \(U_i\), pushing forward along the structure morphisms yields
an isomorphism of graded \(\mathcal O_{U_i}\)-algebras
\[
\Phi_i:\;
 R_{\beta,U_i}
:=
\bigoplus_{n\ge 0} (\pi_{\beta})_{*}(\mathcal L_{\beta}^{\otimes n})|_{U_i}
\;\xrightarrow{\ \sim\ }\;
 R_{T, U_i}
:=
\bigoplus_{n\ge 0} (\pi_{T})_{*}(\mathcal L_{T}^{\otimes n})|_{U_i}.
\]
Equivalently, \(\widetilde{\alpha}\) induces an isomorphism of the relative affine cones
\[
\Spec_{U_i}( R_{\beta,U_i})
\;\xrightarrow{\ \sim\ }\;
\Spec_{U_i}( R_{T, U_i}).
\]

Now set
\[
A_{U_i}:=\mathcal O_{U_i}[x_0,x_1,x_2,x_3],\qquad
\deg x_0=\deg x_1=\deg x_2=1,\  \deg x_3=3.
\]
The two weighted sextic embeddings
\[
i_T:\mathcal X_T|_{U_i}\hookrightarrow \P_{U_i}(1,1,1,3),
\qquad
i_\beta:\mathcal X_\beta|_{U_i}\hookrightarrow \P_{U_i}(1,1,1,3)
\]
give surjective graded \(\mathcal O_{U_i}\)-algebra homomorphisms
\[
q_{T,i}:A_{U_i}\twoheadrightarrow \mathcal R_{T,U_i},
\qquad
q_{\beta,i}:A_{U_i}\twoheadrightarrow \mathcal R_{\beta,U_i}.
\]
By the construction of $\mathcal{X} \subset \P(S_d) \times W$, the kernels of the above maps are generated by sub-line bundles
 \(K_{T, U_i}\subset (A_{U_i})_6\) and \(K_{\beta, U_i}\subset (A_{U_i})_6\) on $U_i$, that are pullbacks of $\cO_{\P(S_6)} (-1)|_{B_2} \subset \pi_* (\mathcal{L}^{\otimes 6})$ corresponding to the universal family $\mathcal{X}.$
After refining the open cover if necessary, we may assume that
\[
K_{T, U_i}=\mathcal O_{U_i}\cdot f_i,\qquad
K_{\beta, U_i}=\mathcal O_{U_i}\cdot g_i
\]
for some homogeneous elements
$f_i,\ g_i\in \Gamma ((A_{U_i})_6).$
Then
\[
R_{T,U_i}\simeq A_{U_i}/(f_i),
\qquad
R_{\beta,U_i}\simeq A_{U_i}/(g_i),
\]
and \(\Phi_i\) may be viewed as a graded \(\mathcal O_{U_i}\)-algebra isomorphism
\[
\varphi_i:\;
A_{U_i}/(g_i)\xrightarrow{\ \sim\ }A_{U_i}/(f_i).
\]

Choose homogeneous lifts in \(A_{U_i}\) of the images under \(\varphi_i\) of the classes of
\(x_0, x_1, x_2, x_3\). This yields a graded \(\mathcal O_{U_i}\)-algebra endomorphism
$
\widetilde{\varphi}_i:A_{U_i}\rightarrow  A_{U_i}
$
lifting \(\varphi_i\). Applying the same construction to \(\varphi_i^{-1}\), we obtain a graded endomorphism
$
\widetilde{\psi}_i:A_{U_i}\rightarrow A_{U_i}.
$
Since \(\deg f_i=\deg g_i=6\), the quotient maps
$
A_{U_i}\to A_{U_i}/(f_i),
$ and 
$A_{U_i}\to A_{U_i}/(g_i)$
are isomorphisms in degrees \(\le 3\). Hence
$
\widetilde{\psi}_i\circ \widetilde{\varphi}_i
$ and
$
\widetilde{\varphi}_i\circ \widetilde{\psi}_i
$
induce the identity on the degree-\(m\) pieces for every \(m\le 3\).
Since \(A_{U_i}\) is generated by its degree-\(1\) and degree-\(3\) parts, this shows that
$ \widetilde{\varphi}_i$ is a graded automorphism of \(A_{U_i}\).

Taking \(\Proj\), we obtain an automorphism
\[
\gamma_i\in \Aut_{U_i}\!\bigl(\P_{U_i}(1,1,1,3)\bigr)=G(U_i).
\]
By construction, the diagram of graded algebras
\[
\begin{tikzcd}
A_{U_i} \arrow[r,"q_{\beta,i}"] \arrow[d,"\widetilde{\varphi}_i"'] &
\mathcal R_\beta|_{U_i} \arrow[d,"\Phi_i"] \\
A_{U_i} \arrow[r,"q_{T,i}"'] &
\mathcal R_T|_{U_i}
\end{tikzcd}
\]
commutes. Passing to \(\Proj\), this means precisely that
$
i_\beta\circ \alpha
=
\gamma_i\circ i_T
$
as morphisms
\[
\mathcal X_T|_{U_i}\rightarrow \P_{U_i}(1,1,1,3).
\]
Since the linearization is unique up to $\Gamma (\cO^{\times})$, $\gamma_i$ glue on overlaps $U_i \cap U_j$, and we obtain the desired element
$\gamma_\alpha\in G(T).$
This construction is functorial in \(T\), and it is inverse to the previously defined map
$
G\rightarrow B_2\times_{\mathcal U_2}S.
$
Note that $G=\Aut_{\P_S (1,1,1,3)/S} \simeq \Aut_{\P(1,1,1,3)/k} \times_k S$ is smooth over $S$ (see \cite{LiendoLucchiniArteche}*{Theorem 2} for example).
Therefore,
$\mu \colon B_{2}\rightarrow \mathcal U_2$
is smooth as desired.
The representability follows from the same argument as in \cite{Rizov}*{Theorem~4.3.3}.
\end{proof}

Let $o \in \Spec S_{D}$ be the origin.
Then the natural map $ \Spec S_{D} \setminus o \rightarrow \P(S_d)$ is a smooth surjection.
Let $B'_{2D} \subset \Spec S_{D} \setminus o$ be the pullback of $B_{2D} \subset \P(S_d)$, which has a smooth surjection $\mu'$ onto $\mathcal{U}_{2D} \subset \mathcal{F}_{2D}$ by \cref{prop:smoothsurjection}.

\begin{definition}
Let $h \in \{1,2, \ldots, 11\}$.
We wet
\[
B_{2D,h}' := \mu'^{-1} (\mathcal{F}_{2D,h}) \subset B'_{2D}.
\]
\end{definition}

By the smoothness of $\mu'$, Ogus's theorem \cref{thm:Ogus} implies the following:

\begin{proposition}\label{Ogus-B'}
Let $h \in \{1,2,\ldots, 10,11\}.$
\begin{enumerate}
\item 
The closed substack $B'_{2D,h} \subset B'$ is equi-dimensional locally complete intersection of codimension $h-1$.
\item
Let $b \in B'_{2D,h} (k)$ be a point.
Then the following are equivalent.
\begin{enumerate}
    \item 
The scheme $B'_{2D, h}$ is singular at $b$.
\item 
The K3 surface $\mathcal{X}_{b}$ is supersingular and $\tau(\mathcal{X}_b, \cO_{W} (1)|_\mathcal{X_b}) <h.$
\end{enumerate}
\end{enumerate}
In particular, $B'_{2D,h}$ is reduced if $1 \leq h \leq 10$, and generically non-reduced if $h=11.$
\end{proposition}

\subsection{Artin invariant versus non-splitting index}
Following Notation~\ref{notation:AI-vs-NS}, we set \(S_d = k[s_1, \ldots, s_m]\) in the following.
\begin{proposition}\label{defining-equation-ht}
In the setting of \cref{notation:AI-vs-NS}, we assume $N=3$ and 
\[
(q_0,q_1,q_2,q_3,d)=(1,1,1,1,2)\quad  \text{or}\quad (1,1,1,3,1). 
\]
Then $B'_{2D,h+1}=\bigcap_{i=1}^h V(H_i(s,0)) \cap B'_{2D}$  for every $1 \leq h \leq 9$.
\end{proposition}

\begin{proof}
We set $C_h:=\bigcap_{i=1}^h V(H_i (s,0)) \cap B'_{2D}$.
By \cite{kty}*{Theorem~C} and \cite{Yobuko}*{Theorem~4.5}, we have $B'_{2D,h+1}(k)=C_h(k)$.
Since $B'_{2D,h+1}$ is reduced by \cref{Ogus-B'}, it is enough to show that $C_h$ is reduced.
Since $B'_{2D,h+1}(k)=C_h(k)$ and $B'_{2D,h+1}$ is equi-dimensional of codimension $h$, so is $C_h$.
Since $C_h$ is defined by $h$ equations, it is locally complete intersection and $H_1 (s,0),\ldots,H_h (s,0)$ is a regular sequence on $B'$.
In particular, $Z_h \cap B'$ is the singular locus of $C_h$. 
Now, it is enough to show that the regular locus of $C_h$ is dense.
We take $b \in B'(k)$ such that $\sht(X_b)=h+1$, then $b \in C_h$.
By \cref{ns-vs-qfs} and \cref{tau-vs-ns}, we have $\tau(b)=\ns(G(b))=\infty$, thus $b \notin Z_h$.
Since $Z_h \cap B'$ is the singular locus of $C_h$, the point $b$ is contained in the regular locus of $C_h$.
On the other hand, by \cref{Ogus-B'} (1), the $\sht=h+1$ losus is dense in $B'_{2D,h+1}$, thus so in $C_h$.
Therefore, the regular locus of $C_h$ is dense, as desired.
\end{proof}

\begin{theorem}\label{ns-vs-K3-tau}
Let $k$ be an algebraically closed field of characteristic $p>0$.
Let $(X,L)$ be a polarized K3 surface over $k$ such that either $X$ is a quartic hypersurface in $\P^3_k$ with $L=\cO_X(1)$ or $X$ is a sextic hypersurface in $\P_k(1,1,1,3)$ with $L=\cO_X(1)$.
Let $f \in k[x_0,\ldots,x_3]$ be a defining equation of $X$.
Then we have
\[
\tau(X,L)=
\begin{cases}
\ns(f) & \text{if $\ns(f)\le 9$},\\
10 & \text{if $10 \le \ns (f) < \infty$}, \\
\infty & \text{if $\ns (f) = \infty$.}
\end{cases}
\]
\end{theorem}

\begin{proof}
By \cref{ns-vs-qfs}, we may assume that $X$ is supersingular.
We take $b \in B'_{2D}$ corresponding to $(X,L)$.
Note that, by the proof of \cref{defining-equation-ht}, for $1\leq h \leq 9$, $Z_h$ is the singular locus of $B'_{2d, h+1} (k).$ 
Therefore, by \cref{tau-vs-ns,defining-equation-ht,Ogus-B'}, if $\tau(X,L) \leq 9$ or $\ns(f) \leq 9$, then we have $\tau(X,L)=\ns(f)$.
Thus, we obtain the desired result.
\end{proof}

The following corollary generalizes \cite[Theorem 1.1]{Bhatt-Singh} (cf.\ \cite[Theorem 1]{OgusHasse} and \cite[Proposition 2.4]{ItoBrauer}) in the case of K3 surfaces.
\begin{corollary}
\label{cor:K3Artininvariant1}
In the situation of \cref{ns-vs-K3-tau}, 
we have $\tau (X,L) =1$ if and only if
$f^{p-2} \in \m^{[p]}.$
Moreover, for smooth quartic surfaces (resp.\ sextic surfaces), this is equivalent to $\sigma (X) =1$ if $p \geq 3$ (resp. for any $p$.)
\end{corollary}

\begin{remark}\label{ns-for-RDP-K3}
Let
\[
f = x^4 + y^4 + z^4 + w^4+ x^2y^2 + x^2z^2 + y^2z^2+ xyz(x + y + z) \in A=\overline{\F}_2[x,y,z,w].
\]
Since $\Delta(f) \in \m^{[4]}$, we have $\ns(f)=2$ by \Cref{explicit-formula}.
On the other hand, the Artin invariant of the minimal resolution of $\Proj (A/fA)$ is equal to \(1\) by \cite{DK03}*{Theorem~1.1(vii)}.
Thus, \Cref{ns-vs-K3-tau} does not admit a straightforward analogue for
rational double point K3 surfaces.
\end{remark}

\section{Non-splitting index in mixed characteristic}
\label{section:mixed}

\subsection{Definition and properties of non-splitting index}

We define the non-splitting index by using the notion of splitting-order sequence defined in \cite{Yoshikawa25-cri}.
We note that the splitting-order sequence does not depend on the choice of Frobenius lifts by \cite{Yoshikawa25-cri}*{Theorem~4.9}.

\begin{definition}
Let $(\wt{A},\m)$ be a regular local ring with $p \in \m$.
Let $\wt{f} \in \wt{A}$ be such that $\wt{f},p$ form a regular sequence.
Let $\bs(\wt{f})=(s_i)_{i \geq 0}$ be the splitting-order sequence of $\wt{f}$ defined in \cite{Yoshikawa25-cri}.
We define the \emph{non-splitting index} $\ns(\wt{f})$ of $\wt{f}$ by
\[
\mathrm{ns}(\wt{f}):=
\begin{cases}
\infty & \text{if $s_n \leq 1$ for all $n \geq 1$,} \\
\min\{n \geq 1 \mid s_n \geq 2\} & \text{otherwise.}
\end{cases}
\]
\end{definition}

\begin{proposition}\label{characterization}
Let $(\wt{A},\m)$ be a regular local ring with $p \in \m$ equipped with a finite Frobenius lift $\phi$.
Let $\wt{f} \in \wt{A}$ be such that $\wt{f},p$ form a regular sequence.
Let $\bs(\wt{f})=(s_i)_{i \geq 0}$ be the splitting-order sequence of $\wt{f}$.
Set
\[
\Delta(\wt{f}):=\frac{\wt{f}^p-\phi(\wt{f})}{p}.
\]
Then, for a positive integer $n$, the following conditions are equivalent:
\begin{enumerate}
\item We have $n=\ns(\wt{f})$.
\item We have $s_1=\cdots=s_{n-1}=1$ and $s_n \geq 2$.
\item We have $\wt{f}_i \notin \m^{[p^i]}$ and $\wt{f}\wt{f}_i \in \m^{[p^i]}$ for all $1 \leq i \leq n-1$, and
\[
\wt{f}_n \in \m^{[p^n]},
\]
where
\[
\wt{f}_m:=\wt{f}^{p-2}(\wt{f}^{p^2-2p}\Delta(\wt{f}))^{1+p+\cdots+p^{m-2}}
\]
for every $m \geq 1$.
\end{enumerate}
In particular, if $\ns(\wt{f})<\infty$, then $\wt{A}/\wt{f}$ is not quasi-$F$-split.
\end{proposition}

\begin{proof}
We first prove $(1) \Rightarrow (2)$.
Assume that $n=\ns(\wt{f})$.
Then we have $s_1,\dots,s_{n-1} \leq 1$.

Suppose that $s_i=0$ for some $1 \leq i \leq n-1$.
Then $\wt{A}/\wt{f}$ is quasi-$F$-split by \cite{Yoshikawa25-cri}*{Proposition~3.9}.
Hence, by \cite{Yoshikawa25}*{Theorem~B}, we have
\[
\ppt(\wt{A}/\wt{f},p)>\frac{p-2}{p-1}.
\]
Therefore, by \cite{Yoshikawa25-cri}*{Theorem~B}, we obtain $s_i \leq 1$ for every $i \geq 1$.
This contradicts the assumption that $\ns(\wt{f})<\infty$.
Thus, we must have
\[
s_1=\cdots=s_{n-1}=1.
\]
Since $n=\ns(\wt{f})$, it follows by definition that $s_n\geq 2$.
This proves $(1) \Rightarrow (2)$ and the final assertion.

The implication $(2) \Rightarrow (1)$ is immediate from the definition of $\ns(\wt{f})$.

The equivalence $(2) \Leftrightarrow (3)$ follows from
\cite{Yoshikawa25-cri}*{Corollary~5.3}. Here we note that
$\wt{f}_m= \wt{f}(1,\ldots,1,2)$ in the notation of \cite{Yoshikawa25-cri}*{Corollary~5.3}.
\end{proof}

\begin{proposition}\label{characterization-non-qfs}
We use the same notation in \Cref{characterization}.
Assume $\wt{A}/\wt{f}$ is not quasi-$F$-split.
Then we have
\[
\ns(\wt{f})=\inf\{n \mid \wt{f}_n \in \m^{[p^n]}\},
\]
where we set $\inf \emptyset:=\infty$.
\end{proposition}

\begin{proof}
First, we assume $\ns(\wt{f})=\infty$.
Since $\wt{A}/\wt{f}$ is not quasi-$F$-split, by \cite{Yoshikawa25-cri}*{Proposition~3.9}, we have $s_i=1$ for every $i \geq 1$.
Thus, it implies that $\wt{f}_n \notin \m^{[p^n]}$. 

Next, we assume $n:=\ns(\wt{f})<\infty$.
By \Cref{characterization}, we have $\wt{f}_i \notin \m^{[p^i]}$ and $\wt{f}_n \in \m^{[p^n]}$, as desired.
\end{proof}

\begin{proposition}\label{explicit-n}
Let $(\wt{A},\m)$ be a regular local ring with $p \in \m$. 
Let $\wt{f} \in \wt{A}$ such that $\wt{f},p$ form a regular sequence.
Then $\wt{A}[T]/(T^{l}+\wt{f})$ is not quasi-$F$-split for $l \geq p^{\ns(\wt{f})}$.
\end{proposition}

\begin{proof}
Set $n:=\ns(\wt{f})$.
Fix an integer $l \geq p^n$ and set $g:=T^{l}+\wt{f}$.
Set $g_1:=g^{p-2}$, $\wt{f}_1:=\wt{f}^{p-2}$,
\[
g_i:=g^{p-2}(g^{p^2-p}\Delta(g))^{1+p+\cdots+p^{i-2}}, \quad \text{and} \quad \wt{f}_i:=\wt{f}^{p-2}(\wt{f}^{p^2-p}\Delta(\wt{f}))^{1+p+\cdots+p^{i-2}} 
\]
for $i \geq 2$.
Then $g_i \equiv f_i \pmod{(T^{p^n},p)}$.
By \Cref{characterization}, we have $\ns(g)=\ns(\wt{f})=n< \infty$.
By \Cref{characterization}, the ring $A[T]/(T^{l}+\wt{f})$ is not quasi-$F$-split, as desired.
\end{proof}


\subsection{Comparison of non-splitting indices in positive and mixed characteristic}

In this subsection, we use the notation introduced in \cref{notation:AI-vs-NS}.

\begin{lemma}\label{decomposition-G_n-s-t}
For $n\ge 1$ and $1 \leq i \leq m$, we have the following decomposition:
\[
G_{n,i}(s,t)
=
G_{n,i}(s,0)
-
\sum_{j=1}^{n-1}
M_j(s,t)^{p^{n-j}}G_{n-j,i}(s,t).
\]
\end{lemma}

\begin{proof}
For $n=1$, the equality $G_{1,i}(s,t)=G_{1,i}(s,0)$ follows directly from the definition.

Let $n\ge 2$.
By definition, we have
\begin{align}\label{eq:1-G_n}
G_n(s,t)
&=(G(s)^{p-2})^{p^{n-1}}
(\Delta(G)(s,0)-G(t)^p)^{p^{n-2}}
G_{n-1}(s,t) \notag \\
&=(G(s)^{p^2-2p}\Delta(G)(s,0))^{p^{n-2}}
G_{n-1}(s,t) \notag \\
&\qquad -(G(s)^{p-2}G(t))^{p^{n-1}}
G_{n-1}(s,t). 
\end{align}

The the coeficient of $x^{p^n-1-\alpha_i}$ in second term in \eqref{eq:1-G_n} is
\[
-M_1(s,t)^{p^{n-1}}G_{n-1,i}(s,t)x^{p^n-1}.
\]
If $n=2$, the first term in \eqref{eq:1-G_n} is
\[
G(s)^{p^2-2p}\Delta(G)(s,0)G(s)^{p-2}
=
G_2(s,0).
\]
Therefore,
\[
G_{2,i}(s,t)=G_{2,i}(s,0) - M_1(s,t)^pG_{1,i}(s,t).
\]

Now assume $n\ge 3$.
The first term in \eqref{eq:1-G_n} can be written as
\begin{align*}
&(G(s)^{p^2-2p}\Delta(G)(s,0))^{p^{n-2}}
G_{n-1}(s,t) \\
&=(G(s)^{p^2-2p}\Delta(G)(s,0))^{p^{n-2}+p^{n-3}}
G_{n-2}(s,t) \\
&\quad - (G(s)^{p^2-2p}\Delta(G)(s,0)G(s)^{p-2}G(t))^{p^{n-2}}
G_{n-2}(s,t).
\end{align*}
The second term in the above expression satisfies
\begin{align*}
&-(G(s)^{p^2-2p}\Delta(G)(s,0)G(s)^{p-2}G(t))^{p^{n-2}}
G_{n-2}(s,t) \\
&=-(G_2(s,0)G(t))^{p^{n-2}}
G_{n-2}(s,t).
\end{align*}
The coefficient of $x^{p^n-1 - \alpha_i}$ in the above element is
\[
-M_2(s,t)^{p^{n-2}}G_{n-2,i}(s,t).
\]

Repeating this procedure inductively, we obtain
\[
G_{n,i}(s,t)
=
G_{n,i}(s,0)
-
\sum_{j=1}^{n-1}
M_j(s,t)^{p^{n-j}}G_{n-j,i}(s,t).
\]
\end{proof}

\begin{lemma}\label{decomposition-G_n}
We have the following decomposition:
\[
G_{n,i}(s,t)
=
\sum_{j=0}^{n-1}
K^{(n)}_j(M_1,\ldots,M_j)\cdot G_{n-j,i}(s,0),
\]
where $K^{(n)}_0=1$, and for $1\le j\le n-1$, $K^{(n)}_j$ is a polynomial in
$M_1(s,t),\ldots,M_j(s,t)$ that is independent of $i$, satisfying
\[
K^{(n)}_j(M_1,\ldots,M_j)
=
-M_j^{p^{n-j}}
+
(\text{a polynomial in } M_1,\ldots,M_{j-1}).
\]
\end{lemma}

\begin{proof}
By \cref{decomposition-G_n-s-t}, we have
\[
G_{n,i}(s,t)
=
G_{n,i}(s,0)
-
\sum_{j=1}^{n-1}
M_j(s,t)^{p^{n-j}}G_{n-j,i}(s,t).
\]
Applying \cref{decomposition-G_n-s-t} recursively to each $G_{n-j,i}(s,t)$
in the right-hand side, we can express $G_{n-j,i}(s,t)$ as a linear combination
of $G_{n-j-l,i}(s,0)$ with coefficients given by polynomials in
$M_1(s,t),\ldots,M_{n-j-1}(s,t)$.
Collecting the coefficients of each $G_{n-j,i}(s,0)$, we obtain the asserted
decomposition.
The description of the leading term of $K^{(n)}_j$ follows directly from the
first summand in the above expression.
\end{proof}

\begin{lemma}\label{choice-lambda_i-G}
Let $n \geq 2$.
Choose $b=(b_1,\ldots,b_m)\in k^m$ such that both $\mathcal{G}_n(b,0)$ and $\mathcal{G}_{n-1}(b,0)$ have rank $n-1$.
Then there exist $\lambda_1,\ldots,\lambda_{n-1}\in k$ such that
\[
\begin{pmatrix}
K^{(n)}_{n-1}(\lambda_1,\ldots,\lambda_{n-1}) &
\cdots &
K^{(n)}_{1}(\lambda_1) &
1
\end{pmatrix}
\cdot
\mathcal{G}_n(b,0)
=0.
\]
\end{lemma}

\begin{proof}
Since both $\mathcal{G}_n(b,0)$ and $\mathcal{G}_{n-1}(b,0)$  have rank $n-1$, there exist $c_1,\ldots,c_{n-1}\in k$ such that
\[
\begin{pmatrix}
c_{n-1} &
c_{n-2} &
\cdots &
c_1 &
1
\end{pmatrix}
\cdot
\mathcal{G}_n(b,0)
=0.
\]

We construct $\lambda_1,\ldots,\lambda_{n-1}$ inductively.
By \cref{decomposition-G_n}, the polynomial $K^{(n)}_1(T)-c_1\in k[T]$ is non-constant.
Since $k$ is algebraically closed, there exists $\lambda_1\in k$ such that
\[
K^{(n)}_1(\lambda_1)=c_1.
\]

Assume that $\lambda_1,\ldots,\lambda_{i-1}$ have been chosen for some $2\le j\le n-1$.
Again by \cref{decomposition-G_n}, the polynomial
\[
K^{(n)}_j(\lambda_1,\ldots,\lambda_{j-1},T)-c_j \in k[T]
\]
is non-constant.
Hence, since $k$ is algebraically closed, there exists $\lambda_j\in k$ such that
\[
K^{(n)}_j(\lambda_1,\ldots,\lambda_{i-1},\lambda_j)=c_j.
\]
Therefore, we obtain $\lambda_1,\ldots,\lambda_{n-1}\in k$ satisfying the desired relation.
\end{proof}

\begin{theorem}\label{compare-ns-positive-mixed}
For $b \in k^m$, we have
\[
\ns(G(b))=\min\{\ns(\wt{G}(b,c)) \mid c \in k^m\}.
\]  
\end{theorem}

\begin{proof}
First, we assume $A/(G(b))$ is quasi-$F$-split.
Then $\ns(G(b))=\infty$ by \Cref{finite-ns-to-non-qfs}.
Furthermore, $\wt{A}/(G(b,c))$ is quasi-$F$-split for every $c \in k^m$ by \cite{Yoshikawa25}*{Proposition~4.6}.
By \Cref{characterization}, we have $\ns(G(b,c))=\infty$.
Thus, we obtain the desired equality.

Next, we assume $A/(G(b))$ is not quasi-$F$-split.
Then, for $c \in k^m$, we have
\[
\ns(\wt{G}(b,c))=\inf\{l \mid \wt{G}(b,c)_l \in (p,x_0^{p^l},\ldots,x_N^{p^l})\},
\]
where
\[
\wt{G}(b,c)_l:=\wt{G}(b,c)^{p-2}(\wt{G}(b,c)^{p^2-2p}\Delta(\wt{G}(b,c)))^{1+p+\cdots+p^{l-2}}
\]
by \Cref{characterization-non-qfs}.
Since the image of $\Delta(\wt{G}(b,c))$ in $A$ is $\Delta(G)(b,c)$, we have
\begin{equation}\label{eq:chara-ns}
\ns(\wt{G}(b,c))=\inf\{l \mid G_l(b,c) \in \m^{[p^l]}\}.
\end{equation}

We further assume that $\ns(G(b)) \geq n+1$.
Then $I_n^{\ns}(G(b)) \notin \m^{[p]}$ by \Cref{Fedder-ns}.
By the proof of \Cref{rank-chara-ns}, there exists $a_1,\ldots,a_{n-1} \in \Ker(u)$ such that
\[
a_{i+1}=u(F_*(\Delta(G)(b,0)a_i)) \quad (1 \leq i \leq n-2), \quad u(F_*(\Delta(G)(b,0)a_{n-1})) \notin \m^{[p]}.
\]
Take $c \in k^m$.
Since $u(F_*\Delta(G)(b,0)a))=u(F_*\Delta(G)(b,c)a)$ for $a \in \Ker(u)$, we have
\[
G_n(b,c) \notin \m^{[p^n]}.
\]
In particular, it follows from \eqref{eq:chara-ns} that $\ns(\wt{G}(b,c)) \geq n+1$.
Therefore, we obtain 
\[
\ns(G(b)) \leq \min\{\ns(\wt{G}(b,c)) \mid c \in k^m\}.
\]

Next, we show the converse inequality. We may assume that $\ns(G(b))<\infty$.
By \cref{rank-chara-ns}, 
$n := \ns (G(b))$ satisfies that  the rank of $\mathcal{G}_i(b,0)$ is $i$ for $1 \leq i \leq n-1$ and the rank of $\cG_n(b,0)=n-1$.
We take $\lambda_1,\ldots,\lambda_{n-1} \in k$ as in \cref{choice-lambda_i-G}.
Since $\mathcal{G}_{n-1}(b,0)$ has rank $n-1$, there exists
$c=(c_1,\ldots,c_m)\in k^m$ such that
\[
\begin{pmatrix}
\lambda_1 \\
\lambda_2 \\
\vdots \\
\lambda_{n-1}
\end{pmatrix}
=
\mathcal{G}_{n-1}(b,0)
\cdot 
\begin{pmatrix}
c_1 \\
c_2 \\
\vdots \\
c_{m}
\end{pmatrix}.
\]
We note that the right-hand side is
\[
\begin{pmatrix}
M_1(b,c) \\
M_2(b,c) \\
\vdots \\
M_{n-1}(b,c)
\end{pmatrix}.
\]
By \cref{decomposition-G_n} and the choice of $\lambda_1,\ldots,\lambda_{n-1}$, we obtain
\begin{align*}
   G_{n,i}(b,c)
&=
\sum_{j=0}^{n-1}
K^{(n)}_j(M_1(b,c),\ldots,M_j(b,c))\cdot G_{n-j,i}(b,0) \\
&=\sum_{j=0}^{n-1}
K^{(n)}_j(\lambda_1,\ldots,\lambda_j)\cdot G_{n-j,i}(b,0) \\
&=0.
\end{align*}
Therefore, we have $G_n(b,c) \in \m^{[p^n]}$, and in particular, we obtain 
\[
\ns(\wt{G}(b,c)) \leq n = \ns(G(b))
\]
by \eqref{eq:chara-ns}, as desired.
\end{proof}

\begin{corollary}[cf.~\cite{TY26}]\label{AI-to-example}
Let $X:=(f=0) \subseteq \P^3$ be a quartic K3 surface over an alegbraic closed field in characteristic $p>0$.
If $X$ is supersingular with Artin invariant $n$, then then ring
\[
k[x,y,z,w,t]/(t^{p^n}+f)
\]
is not quasi-$F$-split.    
\end{corollary}

\begin{proof}
By \Cref{tau-vs-ns}, \Cref{compare-ns-positive-mixed} and \Cref{ns-vs-qfs}, there exists a lift $\wt{f} \in W(k)[x,y,z,w]$ of $f$ such that $n \leq \ns(\wt{f})< \infty$.
By \Cref{explicit-n}, the ring
\[
W(k)[x,y,z,w,t]/(t^{p^n}+\wt{f})
\]
is not quasi-$F$-split.
By \cite{Yoshikawa25}*{Proposition~4.6}, the ring 
\[
k[x,y,z,w,t]/(t^{p^n}+f)
\]
is not quasi-$F$-split, as desired.
\end{proof}

\begin{corollary}\label{no-sigma-geq-3}
Let $k$ be an algebraically closed field in characteristic two.
Let $X$ be a smooth sextic hypersurface in $\mathbb{P}_k(1,1,1,3)$.
If $X$ is supersingular, then $\sigma(X) \geq 3$.
\end{corollary}

\begin{proof}
Let $f \in A:=k[x_0,x_1,x_2,y]$ be an equation of $X$, where $\deg(x_i)=1$ and $\deg(y)=3$.
Let $\wt{f}\in \wt{A}:=W(k)[x_0,x_1,x_2,y]$ be a lift of $f$.
It suffices to show that $\ns(\wt{f}) \geq 3$ by \Cref{compare-ns-positive-mixed}, \Cref{ns-vs-K3-tau}, and \Cref{remark:tau}.
Set
\[
\m:=(x_0,x_1,x_2,y) \subseteq A. 
\]
By \Cref{characterization}, we have $\ns(\wt{f}) \geq 2$.
We set 
\[
f=ay^2+y g_3+g_6
\]
for some $a \in k$ and homogeneous elements $g_3, g_6\in A$ with
\[
\deg(g_3)=3, \quad \deg(g_6)=6.
\]
If $a=0$, then $(0:0:0:1)$ is a singular point of $X$.
Thus, we have $a \neq 0$.
After scaling, we may assume $a=1$. 

Suppose that $\ns (\wt{f}) =2$. By \Cref{characterization}, we have $\Delta(\wt{f}) \in \m^{[4]}$.
For some homogeneous element $h \in A$ of degree $6$, we have
\[
\Delta(\wt{f})=y^3g_3+y^2g_6+y g_3g_6+y^2\Delta(g_3)+\Delta(g_6)+h^2.
\]
Since it is contained in $\m^{[4]}$ and the only $y^3$-term in $\Delta(\wt{f})$ is $y^3g_3$, we have $g_3 \in \m^{[4]}$.
Since $\deg(g_3)=3$, we have $g_3=0$.

We now show that this is impossible if $X$ is smooth.
Indeed,
\[
\frac{\partial  f}{\partial y}=0,
\qquad
\frac{\partial  f}{\partial x_i}=\frac{\partial h_6}{\partial x_i}
\qquad (0\le i\le 2).
\]
Since $\deg(h_6)=6$ and $\operatorname{char}(k)=2$, Euler's relation gives
\[
\sum_{i=0}^2 x_i\frac{\partial h_6}{\partial x_i}=6h_6=0.
\]
Thus the three-tuple
\[
\left(\frac{\partial h_6}{\partial x_0},\frac{\partial h_6}{\partial x_1},
\frac{\partial h_6}{\partial x_2}\right)
\]
defines a global section $s$ of $\Omega^1_{\mathbb P^2}(6)$.
If $s$ is nowhere zero, then we obtain the exact sequence
\[
0\to \mathcal O_{\P^2} \to \Omega^1_{\mathbb P^2}(6) \to Q \to 0
\]
with $Q$ a line bundle. 
Therefore
\[
c(E)=c(\mathcal O_Y)c(Q)=c(Q),
\]
so in particular $c_2(\Omega^1_{\mathbb P^2}(6))=0$.

On the other hand, by the Euler sequence
\[
0\to \Omega^1_{\mathbb P^2}(6)\to \mathcal O_{\mathbb P^2}(5)^{\oplus 3}
\to \mathcal O_{\mathbb P^2}(6)\to 0,
\]
we have
\[
c\bigl(\Omega^1_{\mathbb P^2}(6)\bigr)=\frac{(1+5H)^3}{1+6H},
\]
where $H=c_1(\mathcal O_{\mathbb P^2}(1))$. Hence
\[
c_2\bigl(\Omega^1_{\mathbb P^2}(6)\bigr)=21H^2\neq 0.
\]

This contradiction shows that the section $s$
must vanish at some point
\[
[a_0:a_1:a_2]\in \mathbb P^2_{k}.
\]
That is,
\[
\frac{\partial h_6}{\partial x_i}(a_0,a_1,a_2)=0
\qquad (0\le i\le 2).
\]

Since $k$ is algebraically closed, there exists $b\in k$ such that
\[
b^2=h_6(a_0,a_1,a_2).
\]
Then
\[
[a_0:a_1:a_2:b]\in X
\]
is a singular point of $X$, contradicting the smoothness of $X$.

This contradiction shows that $\ns(\wt{f}) \geq 3$, thus we obtain $\ns(f) \geq 3$.
\end{proof}




\subsection{Possibility of non-splitting indices of lifts}

\begin{theorem}\label{possible-ns}
Assume the setting of \cref{notation:AI-vs-NS}, and fix \(b\in k^m\).
Then
\[
\{\ns(\wt{G}(b,c))\mid c\in k^m\}=
\begin{cases}
\{1\} & \text{if }\ns(G(b))=1,\\
\{\infty\} & \text{if }\ns(G(b))=\infty,\\
\{\ns(G(b)),\infty\} & \text{if $1< \ns(G(b))<\infty$}.
\end{cases}
\]
\end{theorem}

\begin{proof}
In the case of \(\ns(G(b))=\infty\), the assertion follows from \Cref{compare-ns-positive-mixed}.

We assume \(\ns(G(b))<\infty\).
Then \(\wt{A}/(\wt{G}(b,c))\) is not quasi-$F$-split for every \(c\in k^m\) by
\Cref{ns-vs-qfs} and \cite{Yoshikawa25}*{Proposition~6.8}.

Let \(V\) be the \(k\)-vector space of homogeneous polynomials of degree \(d\), and set
\[
f:=G(b)\in V.
\]
We use the notation \(v_f\), \(\lambda\), \(u\), and \(T\) as in Section~\ref{section:explicit-formula},
with \(M_i:=x^{\alpha_i}\).

For each \(c\in k^m\), let \(T_c\) be the \(m\times m\) matrix representing
\[
F_*V\to V,\qquad
h\longmapsto u\!\left(F_*\bigl(\Delta(G)(b,c)G(b)^{p-2}h\bigr)\right)
\]
with respect to the bases \(F_*M_1,\dots,F_*M_m\) and \(M_1,\dots,M_m\).
Then \(T_0=T\).

For each \(c\in k^m\), define row vectors \(R_{c,n}\in k^{1\times m}\) recursively by
\[
R_{c,1}:=F(\lambda),\qquad
R_{c,n+1}:=F(R_{c,n})T_c\qquad (n\ge1).
\]
By an argument in \Cref{explicit-formula}, the \(n\)-th row of \(\mathcal G_n(b,c)\)
is \(R_{c,n}\). 
In particular, by \Cref{characterization-non-qfs}, we have
\begin{equation}\label{eq:ns-vs-A-corrected}
\ns(\wt{G}(b,c))
=
\inf\{\,n\ge1\mid R_{c,n}=0\,\}.
\end{equation}

Next, we compute \(T_c\).
By \cref{notation:AI-vs-NS}, we have
\[
\Delta(G)(b,c)=\Delta(G)(b,0)-G(c)^p.
\]
For \(h\in V\), let \(v_h\) be the column vector determined by
\[
h=(M_1,\dots,M_m)v_h.
\]
Then
\begin{align*}
(M_1,\dots,M_m)T_cF^{-1}(v_h)
&=
u\!\left(F_*\bigl(\Delta(G)(b,c)G(b)^{p-2}h\bigr)\right)\\
&=
u\!\left(F_*\bigl(\Delta(G)(b,0)G(b)^{p-2}h\bigr)\right)
-
u\!\left(F_*\bigl(G(c)^pG(b)^{p-2}h\bigr)\right)\\
&=
u\!\left(F_*\bigl(\Delta(G)(b,0)G(b)^{p-2}h\bigr)\right)
-
G(c)\,u\!\left(F_*\bigl(G(b)^{p-2}h\bigr)\right)\\
&=
(M_1,\dots,M_m)\bigl(T_0F^{-1}(v_h)-c\,\lambda F^{-1}(v_h)\bigr),
\end{align*}
where we regard \(c\) as a column vector.
Hence
\begin{equation}\label{eq:A_c-A_0-corrected}
T_c=T_0-c\lambda.
\end{equation}

Now fix a positive integer \(n\), and for \(1\le i\le n\), let \(g_i(c)\) be the \(i\)-th row of
\(\mathcal G_n(b,c)\).
By \cref{decomposition-G_n}, for each \(1\le i\le n\) we have
\[
g_i(c)
=
\sum_{j=0}^{i-1}
K_j^{(i)}\bigl(M_1(b,c),\dots,M_j(b,c)\bigr)\,g_{i-j}(0),
\qquad
K_0^{(i)}=1.
\]
Therefore,
\[
\mathcal G_n(b,c)=\mathcal{K}_n(b,c)\mathcal G_n(b,0),
\]
where \(\mathcal{K}_n(b,c)\) is the lower triangular matrix
\[
\mathcal{K}_n(b,c)=
\begin{pmatrix}
1 & 0 & 0 & \cdots & 0\\
K_1^{(2)} & 1 & 0 & \cdots & 0\\
K_2^{(3)} & K_1^{(3)} & 1 & \cdots & 0\\
\vdots & \vdots & \vdots & \ddots & \vdots\\
K_{n-1}^{(n)} & K_{n-2}^{(n)} & K_{n-3}^{(n)} & \cdots & 1
\end{pmatrix},
\]
whose entries are evaluated at
\[
\bigl(M_1(b,c),\dots,M_{n-1}(b,c)\bigr).
\]
In particular, \(\mathcal{K}_n(b,c)\) is invertible, and hence
\begin{equation}\label{eq:rank-independent}
\rank \mathcal G_n(b,c)=\rank \mathcal G_n(b,0)
\qquad\text{for every }n\ge1\text{ and }c\in k^m.
\end{equation}

Assume now that \(\ns(\wt{G}(b,c))<\infty\) for some \(c\in k^m\), and set
\[
N:=\ns(\wt{G}(b,c)).
\]
By \eqref{eq:ns-vs-A-corrected}, we have
\[
R_{c,N}=0,
\qquad
R_{c,N-1}\neq 0.
\]

We claim that
\[
R_{c,1},\,R_{c,2},\,\dots,\,R_{c,N-1}
\]
are linearly independent over \(k\).
Indeed, suppose that
\[
a_1R_{c,1}+a_2R_{c,2}+\cdots+a_{N-1}R_{c,N-1}=0
\qquad (a_i\in k).
\]
Let
\[
\Phi_c\colon k^{1\times m}\to k^{1\times m},\qquad
r\longmapsto F(r)T_c.
\]
Then \(R_{c,n+1}=\Phi_c(R_{c,n})\) for every \(n\ge1\).
Applying \(\Phi_c^{N-2}\) to the above relation, we obtain
\[
a_1^{p^{N-2}}R_{c,N-1}=0,
\]
because \(\Phi_c^{N-2}(R_{c,i})=R_{c,i+N-2}=0\) for every \(i\ge2\).
Since \(R_{c,N-1}\neq 0\), it follows that \(a_1=0\).

Now assume inductively that \(a_1=\cdots=a_{j-1}=0\) for some \(2\le j\le N-1\).
Then
\[
a_jR_{c,j}+a_{j+1}R_{c,j+1}+\cdots+a_{N-1}R_{c,N-1}=0.
\]
Applying \(\Phi_c^{N-1-j}\), we obtain
\[
a_j^{p^{N-1-j}}R_{c,N-1}=0,
\]
because \(\Phi_c^{N-1-j}(R_{c,i})=R_{c,i+N-1-j}=0\) for every \(i\ge j+1\).
Hence \(a_j=0\).
By induction, we conclude that
\[
a_1=\cdots=a_{N-1}=0.
\]
Thus \(R_{c,1},\dots,R_{c,N-1}\) are linearly independent.

Therefore, for every \(n\ge N-1\), the row space of \(\mathcal G_n(b,c)\) has dimension
\(N-1\), since its first \(N-1\) rows are \(R_{c,1},\dots,R_{c,N-1}\), while all later
rows vanish.
In particular,
\[
\rank \mathcal G_n(b,c)=N-1
\qquad (n\ge N-1).
\]
On the other hand, by \eqref{eq:rank-independent}, we have
\[
\rank \mathcal G_n(b,c)=\rank \mathcal G_n(b,0)
\qquad \text{for all } n\ge 1.
\]
Hence the integer \(N-1\) is independent of \(c\), and therefore
\[
\{\ns(\wt{G}(b,c))\mid c\in k^m\}\setminus\{\infty\}
\]
is a singleton.
By \Cref{compare-ns-positive-mixed}, this unique finite value is \(\ns(G(b))\).

Next, we show that \(\ns(\wt{G}(b,c))=\infty\) for some \(c\) if
\(\lambda\neq 0\).
Assume \(\lambda\neq 0\).
We write $\lambda = (\lambda_1, \ldots, \lambda_m).$
Choose \(j\) such that \(\lambda_j\neq 0\), and let
\[
v:=e_j\in k^m
\]
be the \(j\)-th standard basis vector, regarded as a column vector.
Then
\[
F(v)=v
\qquad\text{and}\qquad
F(\lambda)v=F(\lambda_j)\neq 0.
\]

Set
\[
c:=(\lambda v)^{-1}(T_0v-v),
\]
we obtain
\[
T_cv=T_0v-c\,\lambda v=v=F(v).
\]

Since \(R_{c,1}=F(\lambda)\), we have
\[
R_{c,1}v=F(\lambda)(v)=F(\lambda_j)\neq 0.
\]
Assume inductively that \(R_{c,n}v\neq 0\). Then
\[
R_{c,n+1}v
=
F(R_{c,n})T_cv
=
F(R_{c,n})F(v)
=
F(R_{c,n}v)
\neq 0.
\]
Hence
\[
R_{c,n}v\neq 0
\qquad\text{for all }n\ge1.
\]
In particular, \(R_{c,n}\neq 0\) for every \(n\ge1\), and therefore
\[
\ns(\wt{G}(b,c))=\infty
\]
by \eqref{eq:ns-vs-A-corrected}.

Finally, if \(\lambda=0\), then \(F(\lambda)=0\), so \(R_{c,1}=0\) for every \(c\).
Hence
\[
\ns(\wt{G}(b,c))=1
\]
for every \(c\in k^m\) by \eqref{eq:ns-vs-A-corrected}.
Since \(\ns(G(b))=1\) by \Cref{explicit-formula}, this proves the assertion in this case as well.
The above arguments prove the theorem.
\end{proof}

\begin{corollary}\label{p-pure-lift}
Assume the setting of \cref{notation:AI-vs-NS}, and fix \(b\in k^m\) and write $f:=G(b)$.
If $\ns(f)>1$, then there exists a homogeneous lift $\wt{f} \in W(k)[x_0,\ldots,x_N]$ such that $\wt{R}:=W(k)[x_0,\ldots,x_N]/(\wt{f})$ is perfectoid pure with
\[
\ppt(\wt{R},p) \geq \frac{p-2}{p-1}.
\]
\end{corollary}

\begin{proof}
By \cref{possible-ns}, there exists a homogeneous lift $\wt{f} \in W(k)[x_0,\ldots,x_N]$ such that $\ns(\wt{f})=\infty$.
Thus, the splitting order sequence $(s_i)_{i \geq 0}$ of $\wt{f}$ satisfies $s_i \leq 1$ for every $i \geq 0$.
By \cite{Yoshikawa25-cri}*{Theorem~A}, the ring $\wt{R}$ is perfectoid pure and
\[
\ppt(\wt{R},p) \geq \frac{p-2}{p-1},
\]
as desired.
\end{proof}

\section{Examples}\label{Section:example}

\subsection{Calabi--Yau hypersurfaces}
In this subsection, we present several computational examples using computer algebra systems (cf.\ \cite{Takamatsu_code}).

\begin{example}[Smooth quartic K3 surfaces over $\F_2$]\label{example:ch=2}
According to \cite{Degtyarev}*{Remark 7.7}, a smooth quartic surface $(X,\cO(1))$ in characteristic 2 satisfies $\sigma (X) \geq 3$ ($\sigma (X) \geq 2$ also follows from our \cref{cor:K3Artininvariant1}).
Note that $\tau (X,\cO(1))$ may be $\sigma (X)-1$ in this case.
However, if there exists a line on $X$, we have $\tau (X,\cO(1)) = \sigma (X)$ (see \cref{remark:tau}).
We obtain Table~\ref{table:K3 surfacee-p=2} consisting of smooth supersingular quartics over $\F_2$ containing a line $x=w=0$.
We note that we could not find an example with $\tau (X,\cO(1)) =10$.
We tentatively expect that every $(X,\cO(1))$ with $\sigma (X) =10$ satisfies $\tau (X,\cO(1)) =9.$

\begin{table}[ht]\label{table:p=2}
\caption{Artin invariants of smooth supersingular K3 surfaces over $\F_2$}
\centering
\begin{tabular}{|c|l|}
\hline
$\sigma(X)$ &  \hspace{20mm} equation\\
\hline
\hline
$3$ & $x^4 + x^2y^2 + xy^3 + yw^3 + z^3w$ \\
\hline
$4$ & $x^4 +xy^3 + z^3w + yzw^2 + zw^3$ \\
\hline
$5$ & $x^4 + xy^3 + z^3w + y^2w^2 +  yw^3$ \\
\hline
$6$ & $x^4 + x^2y^2 + xy^3 + yz^2w + z^3w + zw^3$ \\
\hline
$7$ & $x^4 + x^3y + xz^3 + y^3w + w^4$ \\
\hline
$8$ &  $x^4 + xy^3 + x^2yw + z^3w + yzw^2 + xw^3$\\
\hline
$9$ & $x^4 + xy^3 + yw^3 + z^3w$ \\
\hline
\end{tabular}
\label{table:K3 surfacee-p=2}
\end{table}
\end{example}

\begin{example}[Smooth quartic K3 surfaces over $\F_3$]\label{example:ch=3}
For any quartic K3 surfaces $(X, \cO(1))$ in characteristic $3$, we have $\tau(X,\cO(1)) = \sigma (X)$ (see \cref{remark:tau}).
We obtain Table~\ref{table:K3 surfacee-p=3} consisting of smooth supersingular K3 surfaces over $\F_3$:

\begin{table}[ht]
\caption{Artin invariants of smooth supersingular K3 surfaces over $\F_3$}
\centering
\begin{tabular}{|c|l|}
\hline
$\sigma(X)$ &  \hspace{45mm} equation\\
\hline
\hline
$1$ & $x^4+y^4+z^4+w^4$ \\
\hline
$2$ & $x^4 + y^3z  + yz^3 + xy^2w +  w^4 $ \\
\hline
$3$ & $x^4 + x^2y^2 + y^3z + yz^3 + w^4$ \\
\hline
$4$ & $x^3y + xy^3 + x^2yz + z^3w + zw^3$ \\
\hline
$5$ & $x^3 y + x^2yz + y^3z + z^3w + xw^3$ \\
\hline
$6$ & $x^4 + x^3y+x^2yz+y^3z+z^3w+xw^3$ \\
\hline
$7$ & $-x^4+x^3y+y^4 -x^2yz + y^3z + x^2zw +z^3w + xw^3$ \\
\hline
$8$ &  $x^4 + x^3y + x^3z + x^2yz - y^3z + y^2z^2 + z^3w + xyw^2 + xw^3$\\
\hline
$9$ & $-x^4 + x^3y + x^3z + x^2yz + y^3z + y^3w -z^3w + xw^3$ \\
\hline
$10$ &  
$x^4 + x^3y + xy^3 + x^2yz + y^3z + x^3w + y^2zw + z^3w + xyw^2 + xw^3$
\\
\hline
\end{tabular}\\
\label{table:K3 surfacee-p=3}
\end{table}
\end{example}

\begin{definition}
Let $k$ be an $F$-finite field of characteristic $p>0$, and let $X$ be a smooth $d$-dimensional hypersurface in $\P^{d+1}$.
Assume that $d\ge 3$.
Let $f$ be a defining equation of $X$ in $\P^{d+1}$.
We define the non-splitting index of $X$ by
\[
\ns(X):=\ns(f).
\]
\end{definition}

\begin{remark}
The above definition is well-defined.
Indeed, let $\sigma\colon X\xrightarrow{\sim} X$ be an automorphism.
Since $d\ge 3$, we have $\rho(X)=1$ by the Grothendieck--Lefschetz theorem; hence $\sigma$ preserves the hyperplane class and extends to an automorphism of the ambient projective space $\P^{d+1}$.
Thus $\sigma$ is induced by an automorphism $\phi$ of $A:=k[x_0,\ldots,x_{d+1}]$.
Since $\phi$ induces an isomorphism of pairs
\[
(A,(1-\frac{2}{p^n})\div(f))
\simeq
(A,(1-\frac{2}{p^n})\div(\phi(f))),
\]
it follows from the definition that $\ns(f)=\ns(\phi(f))$.
\end{remark}

\begin{example}[Quintic threefolds over $\F_2$]\label{example:CY}
For quintic threefolds in $\P^4$ over $\F_2$,
the variety $X$ defined by the following polynomial $f \in k[x,y,z,w,u]$ satisfies $\ns (X) =58.$
Among the examples we have found so far, this is the largest value.
\[
\begin{aligned}
f ={} & x^5 + x^4y + x^2y^3 + y^5 + x^3yz + x^3z^2 + z^5 + x^2y^2w + x^3w^2 + xy^2w^2 \\
&+ yz^2w^2 + w^5  + x^4u + xy^3u + xz^2u^2 + xw^2u^2 + yw^2u^2 + u^5.
\end{aligned}
\]
Note that in \cite[Example 6.3]{kty} we obtained an example of a quintic threefold of height 60.
Since the moduli number of quintic threefolds is $101$, this example suggests that the stratification defined by $\ns$ and the height is not as well-behaved as one might expect, in the sense that the dimension does not necessarily drop by one at each step.
\end{example}

\subsection{Delsarte K3 surfaces}

For K3 surfaces in a class called Delsarte type, \cite{ShiodaDelsarte}, \cite{Goto96}, \cite{Yui}, and \cite{Goto04} gave the formula for heights and Artin invariants under certain assumptions on the characteristic.
In this section, we apply our theorem to the smooth members of this class and remove the assumptions on the characteristic imposed in \cite{Goto96}.

\begin{definition}
\label{defn:Delsarte}
Let $(q_0,\ldots,q_3)=(1,1,1,1)$ or $(1,1,1,3)$.
Let $A=(a_{ij})$ be a $4\times 4$ matrix such that $a_{ij}\in \Z_{\geq 0}$ for every $i,j$ and
\[
\sum_{j=0}^3 a_{ij}q_j=\sum_{j=0}^3 q_j
\]
for each $i=0,\ldots,3$.
Set
\[
f_A:=\sum_{i=0}^3 x_0^{a_{i0}}x_1^{a_{i1}}x_2^{a_{i2}}x_3^{a_{i3}} \in \Z[x_0,x_1,x_2,x_3]
\]
Let $\Z[x_0, \ldots, x_3]$ be a $\Z$-graded ring with $\deg (x_i) =q_i.$
We set 
\[
X_{A} = \Proj (\Z[x_0, \ldots, x_3]/ (f_A)).
\]
Moreover, for any algebraically closed field $k$, we set $X_{A,k} := X_{A}\times_{\Z}k$.
If $X_{A,k}$ is smooth, we say $X_{A,k}$ is a smooth Delsarte surface over $k$ associated with $A$.
\end{definition}


\begin{theorem}[cf.~\cite{Goto96}*{Theorem~2.3},~\cite{Goto04}*{Theorem~3.2}]\label{Delsarte}
Let $k$ be an algebraically closed field of characteristic $p>0$.
In the setting of \cref{defn:Delsarte}, we set
$d := \det (A)$.
Let $A'$ be the adjugate matrix of $A$, and define
\[
(\alpha_0,\alpha_1,\alpha_2,\alpha_3):=(1,1,1,1)A'.
\]
Let
\[
g:=\gcd(\alpha_0,\ldots,\alpha_3,d)
\quad\text{and}\quad
e_A:=d/g.
\]
Assume that $X_{A,k}$ is a smooth Delsarte surface.
Then $p\nmid e_A$, and
\[
\sigma(X_{A,k})=\inf\{n \mid p^n \equiv -1 \pmod{e_A}\}
\]
if the set on the right-hand side is nonempty.
Furthermore, if this set is empty, then the Artin--Mazur height $\sht(X_{A,k})$ is finite and
\[
\sht(X_{A,k})=\inf\{n \mid p^n \equiv 1 \pmod{e_A}\}.
\]
\end{theorem}

\begin{proof}
By \cref{prop:quartic-classificationA,prop:sextic-weighted-classificationB}, we classify all smooth Delsarte K3 surfaces in characteristic zero in Tables \ref{Table:DelsarteA} and \ref{Table:DelsarteB}.
Therefore, \(f_A\) is one of the polynomials appearing in these tables.
If $p\nmid a_{ij}$ for every nonzero entry $a_{ij}$, $p\nmid \sum_{j=0}^3 q_j$, and $p\nmid \det(A)$, then the assertion follows from \cite{Goto96}*{Theorem~2.3} and \cite{Goto04}*{Theorem~3.2}.

In the remaining cases, we determine the prime numbers $p$ for which $X_{A,k}$ is smooth, as recorded in $(\star)$-columns of Tables \ref{Table:DelsarteA} and \ref{Table:DelsarteB}.
Hence, the assertion can be proved by using a computer algebra system (\cite{Takamatsu_code}).
Note that when \(p \neq 2\) or $(q_0,\ldots,q_3)=(1,1,1,3)$, we have \(\tau (X_{A,k}, \cO(1)) =\sigma (X_{A,k})\) by \cref{remark:tau}. 
It remains to consider the case where $p=2$ and
\[
f_A= x_0^4+x_0x_1^3+x_1x_2^3+x_2x_3^3.
\] 
Since this surface contains a line $x_0=x_2=0$, we again have \(\tau (X_{A,k}, \cO(1))=\sigma (X_{A,k}) \) by \cref{remark:tau}.

\end{proof}

\begin{table}[htbp]
\centering
\caption{Smooth Delsarte K3 surfaces for $(q_0,\ldots,q_3)=(1,1,1,1)$.}
\label{Table:DelsarteA}
\begin{tabular}{l|c|c|c}
\hline
$f_A$ & $|\det(A)|$ & $e_A$ & $(\star)$ \\
\hline
$x_0^4+x_1^4+x_2^4+x_3^4$ &
$2^8$ &
4 &
$\varnothing$
\\
$x_0^4+x_1^4+x_2^4+x_2x_3^3$ &
$2^6 \cdot 3$ &
12 &
$\varnothing$
\\
$x_0^4+x_1^4+x_2^3x_3+x_3^3x_2$ &
$2^5 \cdot 3$ &
12 &
$\varnothing$
\\
$x_0^4+x_1^4+x_1x_2^3+x_2x_3^3$ &
$2^4 \cdot 3^2$ &
36 &
$\varnothing$
\\
$x_0^4+x_0x_1^3+x_2^4+x_2x_3^3$ &
$2^4 \cdot 3^2$ &
6 &
$\varnothing$
\\
$x_0^4+x_1^3x_2+x_2^3x_3+x_3^3x_1$ &
$2^4 \cdot 7$ &
4 &
3
\\
$x_0^4+x_0x_1^3+x_2^3x_3+x_3^3x_2$ &
$2^5 \cdot 3$ &
12 &
$\varnothing$
\\
$x_0^3x_1+x_1^3x_0+x_2^3x_3+x_3^3x_2$ &
$2^6$ &
4 &
3
\\
$x_0^4+x_0x_1^3+x_1x_2^3+x_2x_3^3$ &
$2^2 \cdot 3^3$ &
27 &
2
\\
$x_0^3x_1+x_1^3x_2+x_2^3x_3+x_3^3x_0$ &
$2^4 \cdot 5$ &
4 &
3,5
\\
\hline
\end{tabular}
\end{table}

\begin{table}[htbp]
\centering
\caption{Smooth Delsarte K3 surfaces for $(q_0,\ldots,q_3)=(1,1,1,3)$.}
\label{Table:DelsarteB}
\begin{tabular}{l|c|c|c}
\hline
$f_A$ & $|\det(A)|$ & $e_A$ & $(\star)$ \\
\hline
$x_0^6+x_1^6+x_2^6+x_3^2$ &
$2^4 \cdot 3^3$ &
6 &
$\varnothing$
\\
$x_0^6+x_1^6+x_2^5x_1+x_3^2$ &
$2^3 \cdot 3^2 \cdot 5$ &
30 &
$\varnothing$
\\
$x_0^6+x_1^5x_2+x_2^5x_1+x_3^2$ &
$2^5 \cdot 3^2$ &
6 &
5
\\
$x_0^6+x_1^6+x_2^3x_3+x_3^2$ &
$2^3 \cdot 3^3$ &
6 &
$\varnothing$
\\
$x_0^6+x_1^5x_0+x_2^5x_1+x_3^2$ &
$2^2 \cdot 3 \cdot 5^2$ &
50 &
3
\\
$x_0^6+x_1^5x_0+x_2^3x_3+x_3^2$ &
$2^2 \cdot 3^2 \cdot 5$ &
15 &
$\varnothing$
\\
$x_0^6+x_1^5x_2+x_2^3x_3+x_3^2$ &
$2^2 \cdot 3^2 \cdot 5$ &
30 &
$\varnothing$
\\
$x_0^5x_1+x_1^5x_2+x_2^5x_0+x_3^2$ &
$2^2 \cdot 3^2 \cdot 7$ &
6 &
5
\\
$x_0^5x_1+x_1^5x_0+x_2^3x_3+x_3^2$ &
$2^4 \cdot 3^2$ &
6 &
5
\\
$x_0^5x_1+x_1^5x_2+x_2^3x_3+x_3^2$ &
$2 \cdot 3 \cdot 5^2$ &
25 &
2,3
\\
\hline
\end{tabular}
\end{table}

\begin{proposition}\label{prop:quartic-classificationA}
Let $k$ be an algebraically closed field of characteristic $0$.
In the setting of \cref{defn:Delsarte}, we consider the case of quartic surfaces, i.e.\ the case where $q_i=1$ for any $i$.
We assume that $X_{A,k}$ is a smooth Delsarte surface.
Then, up to permutation of the coordinates,
$f_A$ is given by one of Table~\ref{Table:DelsarteA}.
\end{proposition}

\begin{proof}


First, for each \(0\le j\le 3\), there exists a monomial in \(f_A\) in which the exponent of \(x_j\) is at least \(3\).
Indeed, after renumbering variables, it suffices to consider the case \(j=0\).
If \(a_{i0}\le 2\) for every \(i\), then \(X_{A,k}\) is singular at \((1:0:0:0)\) by the Jacobian criterion.

Therefore, for each $j$, at least one monomial $x_j^4$ or $x_j^3 x_i ( i \neq j)$ appears in $f_A$.
Since $f_A$ is a polynomial with 4 terms, after reordering variables, we have the following:
For any $0\leq i\leq 3$, $\prod_{j=0}^3 x_j^{a_{ij}}$ is equal to either $x_i^4$ or $x_i^3x_j$ for some $j\neq i$.

Now we can associate to $f_A$ a directed graph $\Gamma$ on the four vertices $P_0, P_1, P_2, P_3$ as follows:
\begin{itemize}
\item if the monomial $x_j^4$ appears, we draw a loop at $P_j$;
\item if the monomial $x_i x_j^3$ appears with $i\neq j$, we draw an arrow
\[
P_j\rightarrow P_i.
\]
\end{itemize}
We note that, for each $j$, there is exactly one outgoing arrow from $P_j$, where a loop is regarded as an arrow from $P_j$ to itself. Hence, every vertex has out-degree exactly one.

Next, we show that for each vertex of \(\Gamma\), there do not exist two or more non-loop arrows pointing to that vertex.
Suppose, to the contrary, that there exist monomials
\(x_i x_j^3\) and \(x_i x_k^3\) in \(f_A\), where
\(i,j,k \in \{0,1,2,3\}\) are pairwise distinct.
We may assume that $i=0, j=1, k=2$.
In this case, we can show that the point $(0:1:-1:0)$ is a singular point of $X_{A,k}$.
It contradicts the smoothness of $X_{A,k}$.

A connected component consisting of a single vertex with a loop will be called an \emph{isolated loop}.
For \(r\ge 2\), a \emph{cycle} means a connected component of the form
\[
P_{i_1}\to P_{i_2}\to \cdots \to P_{i_r}\to P_{i_1},
\]
and a \emph{chain} means a connected component of the form
\[
\begin{tikzcd}
P_{i_1} \arrow[r] &
P_{i_2} \arrow[r] &
\cdots \arrow[r] &
P_{i_r} \arrow[loop right]
\end{tikzcd}
,
\]
where the last vertex carries a loop.
We define the length of a cycle (resp.\ a chain) by the number $r$ as above.

Then we can easily classify $\Gamma$ (and $f_A$) as follows:

\smallskip

\noindent
\underline{Four connected components.}
$
x_0^4+x_1^4+x_2^4+x_3^4;
$
four isolated loops.

\smallskip

\noindent
\underline{Three connected components.}
\begin{enumerate}
\item
$
x_0^4+x_1^4+x_2^4+x_2x_3^3;
$
one chain of length $2$ and two isolated loops.
\item
$
x_0^4+x_1^4+x_2^3x_3+x_3^3x_2;
$
one cycle of length $2$ and two isolated loops.
\end{enumerate}

\smallskip

\noindent
\underline{Two connected components.}
\begin{enumerate}
\item
$
x_0^4+x_1^4+x_1x_2^3+x_2x_3^3;
$
one chain of length $3$ and one isolated loop.
\item
$
x_0^4+x_0x_1^3+x_2^4+x_2x_3^3;
$
two chains of length $2$.
\item
$
x_0^4+x_1^3x_2+x_2^3x_3+x_3^3x_1;
$
one cycle of length $3$ and one isolated loop.
\item
$
x_0^4+x_0x_1^3+x_2^3x_3+x_3^3x_2;
$
one cycle of length $2$ and one chain of length $2$.
\item
$
x_0^3x_1+x_1^3x_0+x_2^3x_3+x_3^3x_2;
$
two cycles of length $2$.
\end{enumerate}

\smallskip

\noindent
\underline{One connected component.}
\begin{enumerate}
\item
$
x_0^4+x_0x_1^3+x_1x_2^3+x_2x_3^3;
$
one chain of length $4$.
\item
$
x_0^3x_1+x_1^3x_2+x_2^3x_3+x_3^3x_0;
$
one cycle of length $4$.
\end{enumerate}
\smallskip
This completes the proof. \qedhere

\end{proof}

\begin{proposition}\label{prop:sextic-weighted-classificationB}
Let $k$ be an algebraically closed field of characteristic $0$.
In the setting of \cref{defn:Delsarte}, we consider the case of weighted sextic surfaces, i.e.\ the case where $(q_0, q_1, q_2, q_3)=(1,1,1,3)$.
We assume that $X_{A,k}$ is a smooth Delsarte surface.
Then, up to permutation of the coordinates,
$f_A$ is given by one of Table~\ref{Table:DelsarteB}.
\end{proposition}

\begin{proof}
First, the monomial \(x_3^2\) appears in \(f_A\).
Indeed, otherwise \(X_{A,k}\) is singular at \((0:0:0:1)\) by the Jacobian criterion.
Next, for each \(0\le i\le 2\), at least one of the monomials
\[
x_i^6,\qquad x_i^5x_j\ (j\neq i),\qquad x_i^3x_3
\]
appears in \(f_A\).
Indeed, after renumbering variables, it suffices to consider the case \(i=0\).
If none of these monomials appears, then \(X_{A,k}\) is singular at \((1:0:0:0)\) by the Jacobian criterion.

Therefore, we may write
\[
f_A=x_3^2+M_0+M_1+M_2,
\]
where, for each \(j=0,1,2\), the monomial \(M_j\) is one of
\[
x_j^6,\qquad x_ix_j^5\ (j\neq i),\qquad x_j^3x_3.
\]

We now associate to \(f_A\) a directed graph \(\Gamma\) with vertices \(P_0,P_1,P_2,P_3\) as follows.
For each \(j\in\{0,1,2\}\),
\begin{itemize}
\item if \(M_j=x_j^6\), we draw a loop at \(P_j\);
\item if \(M_j=x_ix_j^5\) with \(i\in\{0,1,2\}\) and \(j\neq i\), we draw a directed edge \(P_j\to P_i\);
\item if \(M_j=x_j^3x_3\), we draw a directed edge \(P_j\to P_3\).
\end{itemize}
The monomial \(x_3^2\) gives a loop at \(P_3\).
Thus each of the vertices \(P_0,P_1,P_2\) has exactly one outgoing edge, and \(P_3\) carries a loop.

Next, we show that no vertex of \(\Gamma\) receives two distinct non-loop incoming edges.

First, let \(i\in\{0,1,2\}\), and suppose that there exist monomials
\[
x_i x_j^5 \quad\text{and}\quad x_i x_k^5
\]
in \(f_A\), where \(i,j,k\in\{0,1,2\}\) are pairwise distinct.
After renumbering variables, we may assume that \(i=0\), \(j=1\), and \(k=2\).
Then one checks that \((0:1:-1:0)\) is a singular point of \(X_{A,k}\), a contradiction.

Next, suppose that there exist monomials
\[
x_3x_i^3 \quad\text{and}\quad x_3x_j^3
\]
in \(f_A\) with \(i,j\in\{0,1,2\}\) and \(i\neq j\).
After renumbering variables, we may assume that \(i=0\) and \(j=1\).
Then \((1:-1:0:0)\) is a singular point of \(X_{A,k}\), again a contradiction.

We define isolated loops, cycles, and chains, as well as the lengths of cycles and chains, as in the proof of \Cref{prop:quartic-classificationA}.
Then each connected component of \(\Gamma\) is one of the following:
\begin{itemize}
\item an isolated loop;
\item a cycle contained in \(\{P_0,P_1,P_2\}\);
\item a chain contained in \(\{P_0,P_1,P_2\}\) whose terminal vertex carries a loop;
\item a chain ending at \(P_3\).
\end{itemize}

Then we can easily classify \(\Gamma\) (and \(f_A\)) as follows.

\smallskip

\noindent
\underline{Four connected components.}
$
x_0^6+x_1^6+x_2^6+x_3^2;
$
four isolated loops.

\smallskip

\noindent
\underline{Three connected components.}
\begin{enumerate}
\item
$
x_0^6+x_1^6+x_2^5x_1+x_3^2;
$
one chain of length $2$ among $P_0,P_1,P_2$, together with one isolated loop among $P_0,P_1,P_2$ and the isolated loop $P_3$.
\item
$
x_0^6+x_1^5x_2+x_2^5x_1+x_3^2;
$
one cycle of length $2$ among $P_0,P_1,P_2$, together with one isolated loop among $P_0,P_1,P_2$ and the isolated loop $P_3$.
\item
$
x_0^6+x_1^6+x_2^3x_3+x_3^2;
$
one chain of length $2$ ending at $P_3$, together with two isolated loops among $P_0,P_1,P_2$.
\end{enumerate}

\smallskip

\noindent
\underline{Two connected components.}
\begin{enumerate}
\item
$
x_0^6+x_1^5x_0+x_2^5x_1+x_3^2;
$
one chain of length $3$ among $P_0,P_1,P_2$, together with the isolated loop $P_3$.
\item
$
x_0^6+x_1^5x_0+x_2^3x_3+x_3^2;
$
one chain of length $2$ among $P_0,P_1,P_2$ and one chain of length $2$ ending at $P_3$.
\item
$
x_0^6+x_1^5x_2+x_2^3x_3+x_3^2;
$
one isolated loop among $P_0,P_1,P_2$ and one chain of length $3$ ending at $P_3$.
\item
$
x_0^5x_1+x_1^5x_2+x_2^5x_0+x_3^2;
$
one cycle of length $3$ among $P_0,P_1,P_2$, together with the isolated loop $P_3$.
\item
$
x_0^5x_1+x_1^5x_0+x_2^3x_3+x_3^2;
$
one cycle of length $2$ among $P_0,P_1,P_2$ and one chain of length $2$ ending at $P_3$.
\end{enumerate}

\smallskip

\noindent
\underline{One connected component.}
$
x_0^5x_1+x_1^5x_2+x_2^3x_3+x_3^2:
$
one chain of length $3$ ending at $P_3$.

\smallskip
This completes the proof. \qedhere
\end{proof}

\bibliographystyle{skalpha}
\bibliography{bibliography.bib}

@article{Goto96,
  author    = {Yasuhiro Goto},
  title     = {The {A}rtin invariant of supersingular weighted {D}elsarte {K}3 surfaces},
  journal   = {J. Math. Kyoto Univ.},
  volume    = {36},
  number    = {2},
  pages     = {359--363},
  year      = {1996},
  doi       = {10.1215/kjm/1250518553},
}

@article{TY26,
  author  = {Teppei Takamatsu and Shou Yoshikawa},
  title   = {Non-quasi-{$F$}-split canonical affine fourfolds in any characteristic},
  journal = {arXiv preprint arXiv:2602.14792},
  year    = {2026}
}

@article{OTY25,
  author  = {Hirotaka Onuki and Teppei Takamatsu and Shou Yoshikawa},
  title   = {Quasi-{$F$}-splitting and smooth weak del Pezzo surfaces in mixed characteristic},
  journal = {arXiv preprint arXiv:2510.19308},
  year    = {2025}
}

@article{kty,
  title={Fedder type criteria for quasi-${F}$-splitting {I}},
  author={Kawakami, Tatsuro and Takamatsu, Teppei and Yoshikawa, Shou},
  journal={arXiv preprint arXiv:2204.10076},
  year={2022, to appear in Amer. J. Math}
}

@article{KTY2,
  title={Fedder type criteria for quasi-${F}$-splitting {II}},
  author={Kawakami, Tatsuro and Takamatsu, Teppei and Yoshikawa, Shou},
  journal={arXiv preprint arXiv:2511.17270},
  year={2025}
}

@article {KTTWYY1,
    AUTHOR = {Kawakami, Tatsuro and Takamatsu, Teppei and Tanaka, Hiromu and
              Witaszek, Jakub and Yobuko, Fuetaro and Yoshikawa, Shou},
     TITLE = {Quasi-{$F$}-splittings in birational geometry},
   JOURNAL = {Ann. Sci. \'Ec. Norm. Sup\'er. (4)},
  FJOURNAL = {Annales Scientifiques de l'\'Ecole Normale Sup\'erieure.
              Quatri\`eme S\'erie},
    VOLUME = {58},
      YEAR = {2025},
    NUMBER = {3},
     PAGES = {665--748},
      ISSN = {0012-9593,1873-2151},
   MRCLASS = {14E05 (13A35 14G17)},
  MRNUMBER = {4962159},
}

@article{DK03,
  author    = {Igor Dolgachev and Shigeyuki Kondō},
  title     = {A Supersingular {K3} Surface in Characteristic 2 and the Leech Lattice},
  journal   = {International Mathematics Research Notices},
  year      = {2003},
  number    = {1},
  pages     = {1--23},
  doi       = {10.1155/S1073792803000018},
  url       = {https://doi.org/10.1155/S1073792803000018}
}

@article{Goto04,
  author  = {Goto, Yasuhiro},
  title   = {A Note on the Height of the Formal {B}rauer Group of a {$K3$} Surface},
  journal = {Canadian Mathematical Bulletin},
  volume  = {47},
  number  = {1},
  year    = {2004},
  pages   = {22--29},
  doi     = {10.4153/CMB-2004-004-7}
}

@article{Yoshikawa25,
  title        = {Computation method for perfectoid purity and perfectoid {BCM\textendash{}regularity}},
  author       = {Shou Yoshikawa},
  journal      = {arXiv preprint arXiv:2502.06108},
  doi          = {10.48550/arXiv.2502.06108},
  year         = {2025},
}

@article{Yoshikawa25-cri,
  title        = {A Criterion for Perfectoid Purity and the Rationality of Thresholds},
  author       = {Shou Yoshikawa},
  journal      = {arXiv preprint arXiv:2510.19319},
  doi          = {10.48550/arXiv.2510.19319},
  year         = {2025},
}

@article {Fed,
    AUTHOR = {Fedder, Richard},
     TITLE = {{$F$}-purity and rational singularity},
   JOURNAL = {Trans. Amer. Math. Soc.},
  FJOURNAL = {Transactions of the American Mathematical Society},
    VOLUME = {278},
      YEAR = {1983},
    NUMBER = {2},
     PAGES = {461--480},
      ISSN = {0002-9947},
   MRCLASS = {13H10 (13D03 14B05)},
  MRNUMBER = {701505},
MRREVIEWER = {D. Kirby},
       DOI = {10.2307/1999165},
       URL = {https://doi.org/10.2307/1999165},
}

@article{Bhatt-Singh,
  title={The {F}-pure threshold of a {C}alabi--{Y}au hypersurface},
  author={Bhatt, Bhargav and Singh, Anurag K},
  journal={Mathematische Annalen},
  volume={362},
  number={1},
  pages={551--567},
  year={2015},
  publisher={Springer}
}

@article {Rizov,
    AUTHOR = {Rizov, Jordan},
     TITLE = {Moduli stacks of polarized {$K3$} surfaces in mixed
              characteristic},
   JOURNAL = {Serdica Math. J.},
  FJOURNAL = {Serdica. Mathematical Journal. Serdika. Matematichesko
              Spisanie},
    VOLUME = {32},
      YEAR = {2006},
    NUMBER = {2-3},
     PAGES = {131--178},
      ISSN = {1310-6600,2815-5297},
   MRCLASS = {14J28 (14D22 14J10)},
  MRNUMBER = {2263236},
MRREVIEWER = {Martin\ C.\ Olsson},
}

@article {Maulik,
    AUTHOR = {Maulik, Davesh},
     TITLE = {Supersingular {K}3 surfaces for large primes},
      NOTE = {With an appendix by Andrew Snowden},
   JOURNAL = {Duke Math. J.},
  FJOURNAL = {Duke Mathematical Journal},
    VOLUME = {163},
      YEAR = {2014},
    NUMBER = {13},
     PAGES = {2357--2425},
      ISSN = {0012-7094,1547-7398},
   MRCLASS = {14J28 (14G17)},
  MRNUMBER = {3265555},
MRREVIEWER = {Stefan\ Schr\"oer},
       DOI = {10.1215/00127094-2804783},
       URL = {https://doi.org/10.1215/00127094-2804783},
}

@article {Yobuko,
    AUTHOR = {Yobuko, Fuetaro},
     TITLE = {Quasi-{F}robenius splitting and lifting of {C}alabi-{Y}au
              varieties in characteristic {$p$}},
   JOURNAL = {Math. Z.},
  FJOURNAL = {Mathematische Zeitschrift},
    VOLUME = {292},
      YEAR = {2019},
    NUMBER = {1-2},
     PAGES = {307--316},
      ISSN = {0025-5874,1432-1823},
   MRCLASS = {14J32 (13F35 14G17)},
  MRNUMBER = {3968903},
MRREVIEWER = {Tyler\ L.\ Kelly},
       DOI = {10.1007/s00209-018-2198-7},
       URL = {https://doi.org/10.1007/s00209-018-2198-7},
}

@incollection {OgusK3,
    AUTHOR = {Ogus, Arthur},
     TITLE = {Singularities of the height strata in the moduli of {$K3$}
              surfaces},
 BOOKTITLE = {Moduli of abelian varieties ({T}exel {I}sland, 1999)},
    SERIES = {Progr. Math.},
    VOLUME = {195},
     PAGES = {325--343},
 PUBLISHER = {Birkh\"auser, Basel},
      YEAR = {2001},
      ISBN = {3-7643-6517-X},
   MRCLASS = {14J28 (14F30)},
  MRNUMBER = {1827026},
MRREVIEWER = {Vasile\ Br\^inz\u anescu},
}

@book{Huybrechts,
  title={Lectures on K3 surfaces},
  author={Huybrechts, Daniel},
  volume={158},
  year={2016},
  publisher={Cambridge University Press}
}

@article {vdGeer-Katsura,
    AUTHOR = {van der Geer, G. and Katsura, T.},
     TITLE = {On a stratification of the moduli of {$K3$} surfaces},
   JOURNAL = {J. Eur. Math. Soc. (JEMS)},
  FJOURNAL = {Journal of the European Mathematical Society (JEMS)},
    VOLUME = {2},
      YEAR = {2000},
    NUMBER = {3},
     PAGES = {259--290},
      ISSN = {1435-9855,1435-9863},
   MRCLASS = {14J28 (11G50)},
  MRNUMBER = {1776939},
MRREVIEWER = {Atanas\ Iliev},
       DOI = {10.1007/s100970000021},
       URL = {https://doi.org/10.1007/s100970000021},
}

@article{Fletcher,
  title={Working with weighted complete intersections},
  author={Iano-Fletcher, Anthony R},
  journal={Explicit birational geometry of},
  volume={3},
  pages={101--173},
  year={2000}
}

@article {LiendoLucchiniArteche,
    AUTHOR = {Liendo, Alvaro and Lucchini Arteche, Giancarlo},
     TITLE = {Automorphisms of products of toric varieties},
   JOURNAL = {Math. Res. Lett.},
  FJOURNAL = {Mathematical Research Letters},
    VOLUME = {29},
      YEAR = {2022},
    NUMBER = {2},
     PAGES = {529--540},
      ISSN = {1073-2780,1945-001X},
   MRCLASS = {14M25 (14J50)},
  MRNUMBER = {4492227},
MRREVIEWER = {Eunjeong\ Lee},
       DOI = {10.4310/mrl.2022.v29.n2.a9},
       URL = {https://doi.org/10.4310/mrl.2022.v29.n2.a9},
}

@incollection {Ogus-supersingularK3crystal,
    AUTHOR = {Ogus, Arthur},
     TITLE = {Supersingular {$K3$}\ crystals},
 BOOKTITLE = {Journ\'ees de {G}\'eom\'etrie {A}lg\'ebrique de {R}ennes
              ({R}ennes, 1978), {V}ol. {II}},
    SERIES = {Ast\'erisque},
    VOLUME = {64},
     PAGES = {3--86},
 PUBLISHER = {Soc. Math. France, Paris},
      YEAR = {1979},
   MRCLASS = {14J25 (14B10 14F30)},
  MRNUMBER = {563467},
MRREVIEWER = {G.\ Horrocks},
}

@incollection {Ogus-Crystalline-Torelli,
    AUTHOR = {Ogus, Arthur},
     TITLE = {A crystalline {T}orelli theorem for supersingular {$K3$}\
              surfaces},
 BOOKTITLE = {Arithmetic and geometry, {V}ol. {II}},
    SERIES = {Progr. Math.},
    VOLUME = {36},
     PAGES = {361--394},
 PUBLISHER = {Birkh\"auser Boston, Boston, MA},
      YEAR = {1983},
      ISBN = {3-7643-3133-X},
   MRCLASS = {14J28 (14F30 14J15)},
  MRNUMBER = {717616},
MRREVIEWER = {G.\ Horrocks},
}

@article {ItoBrauer,
    AUTHOR = {Ito, Kazuhiro},
     TITLE = {Finiteness of {B}rauer groups of {$K3$} surfaces in
              characteristic 2},
   JOURNAL = {Int. J. Number Theory},
  FJOURNAL = {International Journal of Number Theory},
    VOLUME = {14},
      YEAR = {2018},
    NUMBER = {6},
     PAGES = {1813--1825},
      ISSN = {1793-0421,1793-7310},
   MRCLASS = {14J28 (14F22)},
  MRNUMBER = {3827960},
MRREVIEWER = {Giancarlo\ Lucchini Arteche},
       DOI = {10.1142/S1793042118501087},
       URL = {https://doi-org.kyoto-u.idm.oclc.org/10.1142/S1793042118501087},
}

@article {OgusHasse,
    AUTHOR = {Ogus, Arthur},
     TITLE = {On the {H}asse locus of a {C}alabi-{Y}au family},
   JOURNAL = {Math. Res. Lett.},
  FJOURNAL = {Mathematical Research Letters},
    VOLUME = {8},
      YEAR = {2001},
    NUMBER = {1-2},
     PAGES = {35--41},
      ISSN = {1073-2780},
   MRCLASS = {14F40 (14J10 14J32)},
  MRNUMBER = {1825258},
MRREVIEWER = {Elmar\ Grosse-Kl\"onne},
       DOI = {10.4310/MRL.2001.v8.n1.a5},
       URL = {https://doi-org.kyoto-u.idm.oclc.org/10.4310/MRL.2001.v8.n1.a5},
}

@article {Madapusi,
    AUTHOR = {Madapusi Pera, Keerthi},
     TITLE = {The {T}ate conjecture for {K}3 surfaces in odd characteristic},
   JOURNAL = {Invent. Math.},
  FJOURNAL = {Inventiones Mathematicae},
    VOLUME = {201},
      YEAR = {2015},
    NUMBER = {2},
     PAGES = {625--668},
      ISSN = {0020-9910,1432-1297},
   MRCLASS = {14J28 (11G10 14G17 14G35 14K15)},
  MRNUMBER = {3370622},
MRREVIEWER = {G.\ K.\ Sankaran},
       DOI = {10.1007/s00222-014-0557-5},
       URL = {https://doi.org/10.1007/s00222-014-0557-5},
}

@article {Kim-Madapusi,
    AUTHOR = {Kim, Wansu and Madapusi Pera, Keerthi},
     TITLE = {2-adic integral canonical models},
   JOURNAL = {Forum Math. Sigma},
  FJOURNAL = {Forum of Mathematics. Sigma},
    VOLUME = {4},
      YEAR = {2016},
     PAGES = {Paper No. e28, 34},
      ISSN = {2050-5094},
   MRCLASS = {14G35 (11G18)},
  MRNUMBER = {3569319},
MRREVIEWER = {Su-ion\ Ih},
       DOI = {10.1017/fms.2016.23},
       URL = {https://doi.org/10.1017/fms.2016.23},
}

@article {Ito-Ito-Koshikawa,
    AUTHOR = {Ito, Kazuhiro and Ito, Tetsushi and Koshikawa, Teruhisa},
     TITLE = {C{M} liftings of {$K3$} surfaces over finite fields and their
              applications to the {T}ate conjecture},
   JOURNAL = {Forum Math. Sigma},
  FJOURNAL = {Forum of Mathematics. Sigma},
    VOLUME = {9},
      YEAR = {2021},
     PAGES = {Paper No. e29, 70},
      ISSN = {2050-5094},
   MRCLASS = {11G18 (11G15 14G35 14J28)},
  MRNUMBER = {4241794},
MRREVIEWER = {Salim\ Tayou},
       DOI = {10.1017/fms.2021.24},
       URL = {https://doi.org/10.1017/fms.2021.24},
}

@inproceedings{Madapusierratum,
  title={Erratum to appendix to ‘2-adic integral canonical models’},
  author={Pera, Keerthi Madapusi},
  booktitle={Forum of Mathematics, Sigma},
  volume={8},
  pages={e14},
  year={2020},
  organization={Cambridge University Press}
}

@article {Degtyarev,
    AUTHOR = {Degtyarev, Alex},
     TITLE = {Lines in supersingular quartics},
   JOURNAL = {J. Math. Soc. Japan},
  FJOURNAL = {Journal of the Mathematical Society of Japan},
    VOLUME = {74},
      YEAR = {2022},
    NUMBER = {3},
     PAGES = {973--1019},
      ISSN = {0025-5645,1881-1167},
   MRCLASS = {14J28 (14G17 14J27 14N25)},
  MRNUMBER = {4484237},
       DOI = {10.2969/jmsj/81998199},
       URL = {https://doi.org/10.2969/jmsj/81998199},
}

@incollection {Shimada-Zhang,
    AUTHOR = {Shimada, Ichiro and Zhang, De-Qi},
     TITLE = {{$K3$} surfaces with ten cusps},
 BOOKTITLE = {Algebraic geometry},
    SERIES = {Contemp. Math.},
    VOLUME = {422},
     PAGES = {187--211},
 PUBLISHER = {Amer. Math. Soc., Providence, RI},
      YEAR = {2007},
      ISBN = {978-0-8218-4201-0; 0-8218-4201-3},
   MRCLASS = {14J28 (14J17)},
  MRNUMBER = {2296438},
MRREVIEWER = {I.\ Dolgachev},
       DOI = {10.1090/conm/422/08061},
       URL = {https://doi.org/10.1090/conm/422/08061},
}

@article {Rudakov-Shafarevich,
    AUTHOR = {Rudakov, A. N. and {\v{S}}afarevi{\v{c}}, I. R.},
     TITLE = {Supersingular {$K3$}\ surfaces over fields of characteristic
              {$2$}},
   JOURNAL = {Izv. Akad. Nauk SSSR Ser. Mat.},
  FJOURNAL = {Izvestiya Akademii Nauk SSSR. Seriya Matematicheskaya},
    VOLUME = {42},
      YEAR = {1978},
    NUMBER = {4},
     PAGES = {848--869},
      ISSN = {0373-2436},
   MRCLASS = {14J25 (14J10)},
  MRNUMBER = {508830},
MRREVIEWER = {Miles\ Reid},
}

@article {Shimadachar2,
    AUTHOR = {Shimada, Ichiro},
     TITLE = {Rational double points on supersingular {$K3$} surfaces},
   JOURNAL = {Math. Comp.},
  FJOURNAL = {Mathematics of Computation},
    VOLUME = {73},
      YEAR = {2004},
    NUMBER = {248},
     PAGES = {1989--2017},
      ISSN = {0025-5718,1088-6842},
   MRCLASS = {14J28 (14J17 14J27 14Q10)},
  MRNUMBER = {2059747},
MRREVIEWER = {Andreas\ Leopold\ Knutsen},
       DOI = {10.1090/S0025-5718-04-01641-2},
       URL = {https://doi.org/10.1090/S0025-5718-04-01641-2},
}

@article {ShimadaSupersingularodd,
    AUTHOR = {Shimada, Ichiro},
     TITLE = {Supersingular {$K3$} surfaces in odd characteristic and sextic
              double planes},
   JOURNAL = {Math. Ann.},
  FJOURNAL = {Mathematische Annalen},
    VOLUME = {328},
      YEAR = {2004},
    NUMBER = {3},
     PAGES = {451--468},
      ISSN = {0025-5831,1432-1807},
   MRCLASS = {14J28 (14J20)},
  MRNUMBER = {2036331},
MRREVIEWER = {Shigeyuki\ Kondo},
       DOI = {10.1007/s00208-003-0494-x},
       URL = {https://doi.org/10.1007/s00208-003-0494-x},
}

@article {Yui,
    AUTHOR = {Yui, Noriko},
     TITLE = {Formal {B}rauer groups arising from certain weighted {$K3$}
              surfaces},
   JOURNAL = {J. Pure Appl. Algebra},
  FJOURNAL = {Journal of Pure and Applied Algebra},
    VOLUME = {142},
      YEAR = {1999},
    NUMBER = {3},
     PAGES = {271--296},
      ISSN = {0022-4049,1873-1376},
   MRCLASS = {14F22 (14J28)},
  MRNUMBER = {1721096},
MRREVIEWER = {I.\ Dolgachev},
       DOI = {10.1016/S0022-4049(98)00150-9},
       URL = {https://doi.org/10.1016/S0022-4049(98)00150-9},
}

@article {Stienstra,
    AUTHOR = {Stienstra, Jan},
     TITLE = {Formal group laws arising from algebraic varieties},
   JOURNAL = {Amer. J. Math.},
  FJOURNAL = {American Journal of Mathematics},
    VOLUME = {109},
      YEAR = {1987},
    NUMBER = {5},
     PAGES = {907--925},
      ISSN = {0002-9327,1080-6377},
   MRCLASS = {14L05 (11G35 14J28)},
  MRNUMBER = {910357},
MRREVIEWER = {Noriko Yui},
       DOI = {10.2307/2374494},
       URL = {https://doi.org/10.2307/2374494},
}

@article {Shiodaexplicit,
    AUTHOR = {Shioda, Tetsuji},
     TITLE = {An explicit algorithm for computing the {P}icard number of
              certain algebraic surfaces},
   JOURNAL = {Amer. J. Math.},
  FJOURNAL = {American Journal of Mathematics},
    VOLUME = {108},
      YEAR = {1986},
    NUMBER = {2},
     PAGES = {415--432},
      ISSN = {0002-9327,1080-6377},
   MRCLASS = {14J05 (14C22)},
  MRNUMBER = {833362},
MRREVIEWER = {L.\ B\u adescu},
       DOI = {10.2307/2374678},
       URL = {https://doi.org/10.2307/2374678},
}

@article {Shimadachar2algo,
    AUTHOR = {Shimada, Ichiro},
     TITLE = {Supersingular {$K3$} surfaces in characteristic 2 as double
              covers of a projective plane},
   JOURNAL = {Asian J. Math.},
  FJOURNAL = {Asian Journal of Mathematics},
    VOLUME = {8},
      YEAR = {2004},
    NUMBER = {3},
     PAGES = {531--586},
      ISSN = {1093-6106,1945-0036},
   MRCLASS = {14J28},
  MRNUMBER = {2129248},
MRREVIEWER = {M.\ Kh.\ Gizatullin},
       DOI = {10.4310/ajm.2004.v8.n3.a8},
       URL = {https://doi.org/10.4310/ajm.2004.v8.n3.a8},
}

@book {Blass-Lang,
    AUTHOR = {Blass, Piotr and Lang, Jeffrey},
     TITLE = {Zariski surfaces and differential equations in characteristic
              {$p>0$}},
    SERIES = {Monographs and Textbooks in Pure and Applied Mathematics},
    VOLUME = {106},
      NOTE = {With the collaboration of David Joyce, William E. Lang,
              Raymond Hoobler, Joseph Lipman, Marc Levine, Thorston Ekedahl
              and J.~Blass},
 PUBLISHER = {Marcel Dekker, Inc., New York},
      YEAR = {1987},
     PAGES = {viii+441},
      ISBN = {0-8247-7637-2},
   MRCLASS = {14J25 (14J05)},
  MRNUMBER = {879599},
MRREVIEWER = {Toshiyuki\ Katsura},
}

@article {ShimadaModulicurves,
    AUTHOR = {Shimada, Ichiro},
     TITLE = {Moduli curves of supersingular {$K3$} surfaces in
              characteristic 2 with {A}rtin invariant 2},
   JOURNAL = {Proc. Edinb. Math. Soc. (2)},
  FJOURNAL = {Proceedings of the Edinburgh Mathematical Society. Series II},
    VOLUME = {49},
      YEAR = {2006},
    NUMBER = {2},
     PAGES = {435--503},
      ISSN = {0013-0915,1464-3839},
   MRCLASS = {14J28 (14H50 14J10 14Q10 94B27)},
  MRNUMBER = {2243797},
MRREVIEWER = {I.\ Dolgachev},
       DOI = {10.1017/S0013091504000562},
       URL = {https://doi.org/10.1017/S0013091504000562},
}

@article {Pho-Shimada,
    AUTHOR = {Pho, Duc Tai and Shimada, Ichiro},
     TITLE = {Unirationality of certain supersingular {$K3$} surfaces in
              characteristic 5},
   JOURNAL = {Manuscripta Math.},
  FJOURNAL = {Manuscripta Mathematica},
    VOLUME = {121},
      YEAR = {2006},
    NUMBER = {4},
     PAGES = {425--435},
      ISSN = {0025-2611,1432-1785},
   MRCLASS = {14J28 (14M20)},
  MRNUMBER = {2282430},
MRREVIEWER = {Matthias\ Sch\"utt},
       DOI = {10.1007/s00229-006-0045-3},
       URL = {https://doi.org/10.1007/s00229-006-0045-3},
}

@article {Shimadaprojectivemodelchar5,
    AUTHOR = {Shimada, Ichiro},
     TITLE = {Projective models of the supersingular {$K3$} surface with
              {A}rtin invariant 1 in characteristic 5},
   JOURNAL = {J. Algebra},
  FJOURNAL = {Journal of Algebra},
    VOLUME = {403},
      YEAR = {2014},
     PAGES = {273--299},
      ISSN = {0021-8693,1090-266X},
   MRCLASS = {14J28 (14C17)},
  MRNUMBER = {3166075},
MRREVIEWER = {Stefan\ Schr\"oer},
       DOI = {10.1016/j.jalgebra.2013.12.029},
       URL = {https://doi.org/10.1016/j.jalgebra.2013.12.029},
}

@article {Shimada-ZhangKummer,
    AUTHOR = {Shimada, Ichiro and Zhang, De-Qi},
     TITLE = {On {K}ummer type construction of supersingular {$K3$} surfaces
              in characteristic 2},
   JOURNAL = {Pacific J. Math.},
  FJOURNAL = {Pacific Journal of Mathematics},
    VOLUME = {232},
      YEAR = {2007},
    NUMBER = {2},
     PAGES = {379--400},
      ISSN = {0030-8730,1945-5844},
   MRCLASS = {14J28},
  MRNUMBER = {2366360},
MRREVIEWER = {Trygve\ Johnsen},
       DOI = {10.2140/pjm.2007.232.379},
       URL = {https://doi.org/10.2140/pjm.2007.232.379},
}

@article {KatsuraSchutt,
    AUTHOR = {Katsura, Toshiyuki and Sch\"utt, Matthias},
     TITLE = {Zariski {K}3 surfaces},
   JOURNAL = {Rev. Mat. Iberoam.},
  FJOURNAL = {Revista Matem\'atica Iberoamericana},
    VOLUME = {36},
      YEAR = {2020},
    NUMBER = {3},
     PAGES = {869--894},
      ISSN = {0213-2230,2235-0616},
   MRCLASS = {14J28 (14G17 14K02)},
  MRNUMBER = {4109829},
MRREVIEWER = {Yasuhiro\ Goto},
       DOI = {10.4171/rmi/1152},
       URL = {https://doi.org/10.4171/rmi/1152},
}

@misc{Takamatsu_code,
  author       = {Teppei Takamatsu},
  title        = {Macaulay2 scripts for computing ns},
  NOTE = {available at \url{https://sites.google.com/view/teppei-takamatsu/home/scripts}},
}

@article {Ekedahl-vdGeer,
    AUTHOR = {Ekedahl, Torsten and van der Geer, Gerard},
     TITLE = {Cycle classes on the moduli of {K}3 surfaces in positive
              characteristic},
   JOURNAL = {Selecta Math. (N.S.)},
  FJOURNAL = {Selecta Mathematica. New Series},
    VOLUME = {21},
      YEAR = {2015},
    NUMBER = {1},
     PAGES = {245--291},
      ISSN = {1022-1824,1420-9020},
   MRCLASS = {14C17 (14J10 14J28)},
  MRNUMBER = {3300417},
MRREVIEWER = {Christian\ Liedtke},
       DOI = {10.1007/s00029-014-0156-8},
       URL = {https://doi.org/10.1007/s00029-014-0156-8},
}

@article {ShiodaDelsarte,
    AUTHOR = {Shioda, Tetsuji},
     TITLE = {Supersingular {$K3$} surfaces with big {A}rtin invariant},
   JOURNAL = {J. Reine Angew. Math.},
  FJOURNAL = {Journal f\"ur die Reine und Angewandte Mathematik. [Crelle's
              Journal]},
    VOLUME = {381},
      YEAR = {1987},
     PAGES = {205--210},
      ISSN = {0075-4102,1435-5345},
   MRCLASS = {14J28 (14J05)},
  MRNUMBER = {918849},
MRREVIEWER = {Gerd\ Faltings},
       DOI = {10.1515/crll.1987.381.205},
       URL = {https://doi.org/10.1515/crll.1987.381.205},
}

@article {Matsumoto,
    AUTHOR = {Matsumoto, Yuya},
     TITLE = {Inseparable maps on {$W_n$}-valued local cohomology groups of
              nontaut rational double point singularities and the height of
              {K}3 surfaces},
   JOURNAL = {J. Commut. Algebra},
  FJOURNAL = {Journal of Commutative Algebra},
    VOLUME = {15},
      YEAR = {2023},
    NUMBER = {3},
     PAGES = {377--404},
      ISSN = {1939-0807,1939-2346},
   MRCLASS = {14J28 (13A35 14B15 14J17 14L15 14L30)},
  MRNUMBER = {4680627},
MRREVIEWER = {James\ N.\ Brawner},
       DOI = {10.1216/jca.2023.15.377},
       URL = {https://doi.org/10.1216/jca.2023.15.377},
}

@incollection {Shiodalecturenote,
    AUTHOR = {Shioda, Tetsuji},
     TITLE = {Supersingular {$K3$}\ surfaces},
 BOOKTITLE = {Algebraic geometry ({P}roc. {S}ummer {M}eeting, {U}niv.
              {C}openhagen, {C}openhagen, 1978)},
    SERIES = {Lecture Notes in Math.},
    VOLUME = {732},
     PAGES = {564--591},
 PUBLISHER = {Springer, Berlin},
      YEAR = {1979},
      ISBN = {3-540-09527-6},
   MRCLASS = {14J20 (14J25)},
  MRNUMBER = {555718},
MRREVIEWER = {Piotr\ Blass},
}

@article {Ito-supersingularreduction,
    AUTHOR = {Ito, Kazuhiro},
     TITLE = {On the supersingular reduction of {$K3$} surfaces with complex
              multiplication},
   JOURNAL = {Int. Math. Res. Not. IMRN},
  FJOURNAL = {International Mathematics Research Notices. IMRN},
      YEAR = {2020},
    NUMBER = {20},
     PAGES = {7306--7346},
      ISSN = {1073-7928,1687-0247},
   MRCLASS = {14J28 (14F22 14G17)},
  MRNUMBER = {4172684},
MRREVIEWER = {Remke\ Kloosterman},
       DOI = {10.1093/imrn/rny210},
       URL = {https://doi.org/10.1093/imrn/rny210},
}

@article {Yu-Yui,
    AUTHOR = {Yu, Jeng-Daw and Yui, Noriko},
     TITLE = {{$K3$} surfaces of finite height over finite fields},
   JOURNAL = {J. Math. Kyoto Univ.},
  FJOURNAL = {Journal of Mathematics of Kyoto University},
    VOLUME = {48},
      YEAR = {2008},
    NUMBER = {3},
     PAGES = {499--519},
      ISSN = {0023-608X},
   MRCLASS = {11G25 (14G15 14J28)},
  MRNUMBER = {2511048},
MRREVIEWER = {Matthias\ Sch\"utt},
       DOI = {10.1215/kjm/1250271381},
       URL = {https://doi.org/10.1215/kjm/1250271381},
}

@article {Taelman,
    AUTHOR = {Taelman, Lenny},
     TITLE = {K3 surfaces over finite fields with given {$L$}-function},
   JOURNAL = {Algebra Number Theory},
  FJOURNAL = {Algebra \& Number Theory},
    VOLUME = {10},
      YEAR = {2016},
    NUMBER = {5},
     PAGES = {1133--1146},
      ISSN = {1937-0652,1944-7833},
   MRCLASS = {14J28 (11G25 14G15)},
  MRNUMBER = {3531364},
MRREVIEWER = {Sho\ Tanimoto},
       DOI = {10.2140/ant.2016.10.1133},
       URL = {https://doi.org/10.2140/ant.2016.10.1133},
}
\end{document}